\documentclass{amsart}
\usepackage{amsmath, amssymb, amscd, amsthm, amsfonts, amstext,
  amsbsy, enumerate, tikz}

\newcommand{\seqo}{{}^{<\omega}\omega}
\newcommand{\fhi}{\varphi}

\newcommand{\imp}{\Rightarrow}
\renewcommand{\L}{\mathcal{L}}
\newcommand{\G}{\mathbb{G}}
\newcommand{\QQ}{\mathbb{Q}}
\newcommand{\RR}{\mathbb{R}}

\newtheorem{theorem}{Theorem}[section]
\newtheorem{lemma}[theorem]{Lemma}
\newtheorem{corollary}[theorem]{Corollary}
\newtheorem{proposition}[theorem]{Proposition}

\newtheorem{question}[theorem]{Question}

\theoremstyle{definition}
\newtheorem{claim}{Claim}[theorem]
\newtheorem{defin}[theorem]{Definition}

\theoremstyle{remark}
\newtheorem{remark}[theorem]{Remark}

\begin{document}

\title{Invariantly universal analytic quasi-orders}
\date{April 29, 2011}
\author{Riccardo Camerlo}
\address{Dipartimento di Matematica,
Politecnico di Torino,
corso Duca degli Abruzzi 24,
10129 Torino,
Italy}
\email{camerlo@calvino.polito.it}

\author{Alberto Marcone}
    \address{Dipartimento di Matematica e Informatica,
    Universit\`{a} di Udine,
    viale delle Scienze 206,
    33100 Udine,
    Italy}
\email{alberto.marcone@dimi.uniud.it}

\author{Luca Motto Ros}
\address{Kurt G\"odel Research Center for Mathematical Logic\\
  University of Vienna \\ W\"ahringer Stra{\ss}e 25 \\
 A-1090 Vienna \\
Austria}
\curraddr{Albert-Ludwigs-Universit\"at Freiburg \\
Mathematisches Institut -- Abteilung f\"ur Mathematische Logik\\
Eckerstra{\ss}e, 1 \\
 D-79104 Freiburg im Breisgau\\
Germany}
\email{luca.motto.ros@math.uni-freiburg.de}
\thanks{Camerlo's research was partially supported by FWF (Austrian Research
  Fund) through Project number P 19898-N18. Motto Ros' research was supported by FWF  through Project number P 19898-N18. Marcone's research was partially supported by FWF through Project number P 19898-N18 and by PRIN of Italy.}
\keywords{Analytic equivalence relations, analytic quasi-orders, Borel reducibility, completeness, invariant universality, colored linear orders, dendrites, (ultrametric) Polish spaces, separable Banach spaces}
\subjclass[2010]{03E15}

\begin{abstract}
We introduce the notion of an invariantly universal pair $(S,E)$ where $S$ is an analytic quasi-order and $E\subseteq S$ is an analytic equivalence relation.
This means that for any analytic quasi-order $R$ there is a Borel set $B$ invariant under $E$ such that $R$ is Borel bireducible with the restriction of $S$ to $B$.
We prove a general result giving a sufficient condition for invariant universality, and we demonstrate several applications of this theorem by showing that the phenomenon of invariant universality is widespread.
In fact it occurs for a great number of complete analytic quasi-orders, arising in different areas of mathematics, when they are paired with natural equivalence relations.
\end{abstract}

\maketitle

\section{Introduction}

Given analytic quasi-orders $R,R'$ on standard Borel spaces $X,X'$,
respectively, say that $R$ \emph{Borel reduces} to $R'$, in symbols
$R \leq_B R'$, if there is a Borel function $f \colon X \to X'$ such that
$\forall x,y \in X  (x\, R\, y \iff f(x)\, R'\, f(y))$.
Quasi-orders $R,R'$ are \emph{Borel bireducible} if $R \leq_B R' \leq_B
R$, in symbols $R \sim_B R'$.
In the sequel we will denote by $E_R$ the equivalence relation $R\cap R^{-1}$ induced by a quasi-order $R$.

In \cite{louros}, Louveau and Rosendal proved the existence of a \emph{complete}\label{completeness} (sometimes called
\emph{universal}) analytic quasi-order $S$ on a standard Borel space: this means that $R \leq_B S$ for every analytic
quasi-order $R$ defined on a standard Borel space.
Examples of such $S$ can be given in spaces of countable structures
(for more details on these spaces, a reference is
\cite[\S 16.C]{Kechris1995}): if $\L$ is a countable relational language, let $Mod_\L$ be
the Polish space of (codes for) $\L$-structures with
universe $\omega$, and $j_\L$ the
logic action of $S_\infty$ (the symmetric group of $\omega$) on
$Mod_\L$.
If $ \fhi$ is a sentence of $\L_{\omega_1 \omega}$, then $Mod_\fhi$
will stand for the set of elements of $Mod_\L$ satisfying $\fhi$, and
by Lopez-Escobar theorem (see \cite[Theorem 16.8]{Kechris1995}), the sets
of the form $Mod_\fhi$ are exactly the Borel subsets of $Mod_\L$
invariant under $j_\L$, i.e.\ invariant under isomorphism.
An example of a complete analytic quasi-order $S$ is then the relation
of embeddability for graphs on $\omega$.

In \cite{FriMot} Friedman and the third author strengthened this
result by showing that for
any analytic quasi-order $R$ on a standard Borel space there is a
Borel class $B$ of graphs on $\omega$ invariant under isomorphism
such that $R$ is Borel bireducible with the relation of embeddability
restricted to $B$.
This situation suggests the following definition.

\begin{defin}
Let $S$ be an analytic quasi-order on some standard Borel space $X$
and let $E$ be an analytic equivalence subrelation of $S$ (so that $E\subseteq E_S$).
We say that the pair $(S,E)$ is \emph{invariantly universal} for
analytic quasi-orders if
for any analytic quasi-order $R$ there is a Borel subset $B \subseteq
X$ invariant with respect to $E$ such that the restriction of $S$ to
$B$ is Borel bireducible with $R$.

Similarly, if $F,E$ are analytic equivalence relations on $X$ such
that $E\subseteq F$, the pair $(F,E)$ is \emph{invariantly universal}
for analytic equivalence relations if for every analytic equivalence
relation $R$ there is a Borel $E$-invariant subset $B\subseteq X$ such
that the restriction of $F$ to $B$ is Borel bireducible with $R$.
\end{defin}

Invariant universality is a natural strengthening of completeness: in fact, if $(S,E)$ is a pair which is invariantly universal for the class of analytic quasi-orders, then $S$ must be complete for the same class (and similarly for pairs of the form $(F,E)$). Despite the fact that invariant universality looks as a very strong notion, the results of this paper show that it is a widespread phenomenon. Moreover, it is somehow related to the notions of completeness with respect to some natural strengthenings of Borel reducibilty, called faithful Borel reducibility and classwise Borel embeddability (introduced respectively in \cite{frista1989} and \cite{mottoros}). For example, if $F$ is complete with respect to classwise Borel embeddability for the class of all analytic equivalence relations, then the pair $(F,F)$ is invariantly universal for the same class.

It should also be noted here that invariant universality is an example of
 the descriptive set theoretical phenomenon of universality, as
 opposed to the notion of completeness. Generally speaking, if
 $\Gamma$ is a class of mathematical objects and $X \in \Gamma$, we
 say that $X$ is $\Gamma$-complete if each element of
 $\Gamma$ is in some specific sense reducible to $X$, while we say
 that $X$ is $\Gamma$-universal if it contains in a
 ``natural'' way a copy of every element of $\Gamma$. Let now
 $\Gamma$ be the class of all analytic quasi-orders (or,
 respectively, of all analytic equivalence relations). If we use
 $\leq_B$ as the reducibility notion, then $\Gamma$-completeness coincides with the
 notion of completeness given on page \pageref{completeness}. On the other hand,
in most practical cases one could arguably view the property of ``being Borel
bireducible with the restriction of $S$ to a Borel $E$-invariant set'' as a
translation in our context of the property of ``having a copy naturally contained
in $S$'': with this identification, $\Gamma$-universality coincides
with invariant universality. This will be further clarified by the concrete examples given in Section \ref{sectionapplications}.

Notice that if $(S,E)$ is invariantly universal for quasi-orders, then
$(E_S,E)$ is invariantly universal for equivalence relations. So we
shall confine our attention to the former notion, as all results will
have a counterpart in the equivalence relation setting.

An elementary property to be noted here is that if the pair $(S,E)$ is
universal and $F$ is an equivalence subrelation of $E$, then $(S,F)$
too is universal.

To give a first example of an invariantly universal pair, fix $W\subseteq {}^{\omega }2\times {}^{\omega }2\times {}^{\omega }2$ an analytic set whose sections are exactly all analytic quasi-orders on ${}^{\omega }2$.
Now define
$$(x,y)\, S\, (x',y') \iff x=x'\wedge (x,y,y')\in W;$$
this is the first example of a complete analytic quasi-order presented
in \cite{louros}. Now, the pair $(S,E_S)$ is easily seen to be
invariantly universal. Hence, for every analytic equivalence relation
$E\subseteq E_S$, the pair $(S,E)$ is invariantly universal as well.

In this paper we show that many complete analytic quasi-orders, paired with natural equivalence relations, are invariantly universal.
To see however that this is not always the case, consider an analytic non-Borel subset $A$ of a standard Borel space $X$ and $S$ a complete analytic quasi-order on $X$.
Define the analytic quasi-order $R$ on $X\times X$ by
$$(x,y)\, R\, (x',y')\iff (x,y)=(x',y')\vee (x,x'\in A \wedge y\, S\, y')$$
and the analytic equivalence relation $E$ by
$$(x,y)\, E\, (x',y') \iff y=y'\wedge (x,x'\in A \vee x=x').$$
Then $R$ is a complete analytic quasi-order, but $(R,E)$ (so neither $(R,E_R)$) is not invariantly universal, as the restriction of $R$ to any Borel $E$-invariant set is either non-Borel or a smooth equivalence relation.

The paper consists of two parts. The first one culminates with the
proof of Theorem \ref{theorsaturation}, the main result of the paper,
which gives a widely applicable sufficient condition for a pair $(S,E)$
to be invariantly universal. This condition might appear somewhat
technical, however it turns out to be quite powerful. In fact, it will
be our tool for giving many examples of invariantly universal pairs in
the second part of the article. These include quasi-orders from model
theory, combinatorics, topology, metric space theory and separable
Banach space theory. Notice that each of these results of invariant
universality can also be interpreted as a characterization of the class
of analytic quasi-orders and equivalence relations, as if $(S,E)$ is
invariantly universal (where $S,E$ are defined on a standard Borel
space $X$) then a binary relation $R$ defined on a standard Borel space
is an analytic quasi-order (resp.\ an analytic equivalence relation) if
and only if there is a Borel $E$-invariant  $B \subseteq X$ such that
$R \sim_B {S \restriction B}$ (resp.\ $R \sim_B {E_S \restriction B}$).

We warmly thank Christian Rosendal, to whom we are indebted for several illuminating discussions.
These eventually led to our results on separable Banach spaces contained in Subsection \ref{banach}.

\section{Some preliminaries and notations}

A typical example of a pair $(S,E)$ which will be considered in this
paper is the following: suppose that $A
\subseteq {}^\omega \omega \times X \times X$ --- where ${}^\omega
\omega$ is the Baire space consisting of all functions from $\omega$
into itself --- is Borel and such that for all $x,y,z\in X$
\begin{itemize}
\item $(id, x,x) \in A$,
\item $(f,x,y) \in A
\wedge (g,y,z) \in A \imp (g \circ f, x,z) \in A.$
\end{itemize}
 Then the relation $S = \{ (x,y) \in X^2 \mid \exists f \colon \omega
 \to \omega (f,x,y) \in A \}$ is an analytic
quasi-order, which we will call a \emph{morphism relation}.
The equivalence relation
\[ E = \{ (x,y) \in X^2 \mid
\exists f \in S_\infty ((f,x,y) \in A \wedge (f^{-1},y,x) \in A)\} \]
will be called the $A$-\emph{isomorphism relation} or, with some abuse
of terminology, the $S$-\emph{isomorphism relation}.
This is usually a minor abuse, as in all cases we shall consider there
is a natural $A$ which generates $S$.
The equivalence
relation $E_S$ induced by $S$, which
clearly satisfies $E \subseteq E_S$, will be
called \emph{bi-morphism relation}.

A natural context in which such
morphism relations appear is when considering the quasi-order $S$
induced by some
model theoretic
notions of morphism (like e.g.\ embedding, homomorphism,
weak-homomorphism or epimorphism) on
some $Mod_{\psi }$: in all these cases the $S$-isomorphism relation
associated with $S$ is simply the isomorphism relation $\cong$ so
that, again by Lopez-Escobar theorem, such an
$S$ is invariantly universal if and only if for every analytic quasi-order $R$
there is an $\L_{\omega_1 \omega}$-sentence $\fhi$ such that $R$ is
Borel bireducible with $S$ restricted to $Mod_\fhi$.

The terminology given above will be extended in a natural way to many other kinds
of relations we shall consider, and when dealing with a pair $(S,E)$ given by a
morphism relation and its
associated $S$-isomorphism relation, we will briefly say that
the quasi-order $S$
is invariantly universal.
Using this terminology, we can reformulate some results of \cite{FriMot} (one of which was mentioned in the introduction) by
saying:

\begin{theorem}
The relations of embeddability and of
homomorphism for graphs on $\omega$ are both
invariantly universal.
\end{theorem}

\medskip
We now fix some notation and recall some results for later use.
Given a Polish space $X$, the standard Borel space of closed subsets
of $X$ will be denoted by $F(X)$.
If $X$ is a Polish group, $G(X)$ will stand for the standard Borel
space of closed subgroups of $X$.
The orbit equivalence relation induced by an action $a$ will be denoted $E_a$.

For $t$ a finite sequence, $|t|$ will denote its length.
Given an element $u$ and a natural number $n$, $u^n$ will indicate the
sequence of length $n$ and constant value $u$.
Starting with a linearly ordered set, we shall denote by $\leq_{lex}$ the lexicographic order on the set of finite sequences.
If two sequences $s,t \in \seqo$ have the
same length, we
put $s \leq t$ if $s(i) \leq t(i)$ for every $i < |s|=|t|$, and
define $s+t \in \seqo$ by setting $(s+t)(i) = s(i)+t(i)$. Given a
non-empty set $X$ and a tree
$T$ on $X \times \omega$, we say that $T$ is \emph{normal} if
given $u \in {}^{<\omega}X$ and $s \in \seqo$, $(u,s) \in T$ implies
$(u,t) \in T$ for every $s \leq t$.

It is well-known that any analytic
subset of ${}^\omega 2 \times {}^\omega 2$, in particular any
analytic quasi-order $R$ on ${}^\omega 2$, is the projection $p$ of the body of a tree
on $2 \times 2 \times \omega$.
In \cite[Theorem 2.4]{louros}, this property is strengthened by showing that such an $R$ can be
viewed as the
projection of a normal tree $S$ on $2 \times 2 \times \omega$ such
that the reflexivity and transitivity properties of $R$ are
mirrored by corresponding local properties of $S$. A further
strengthening, useful in applications, was isolated in \cite{FriMot}.
Wrapping up, we have the following.

\begin{proposition}[\cite{louros}, \cite{FriMot}]\label{propnormalform}
  Let $R$ be any analytic quasi-order on ${}^\omega 2$. Then there is a
  normal tree $S$ on $2 \times 2 \times \omega$ such that:
  \begin{enumerate}[i)]
  \item $R = p[S]$;
\item for every $u \in {}^{< \omega}2$ and $s \in \seqo$ of the same
  length, $(u,u,s) \in S$;
\item for every $u,v,w \in {}^{<\omega}2$ and $s,t \in \seqo$ of the
  same length, if $(u,v,s) \in S$ and $(v,w,t) \in S$ then $(u,w,s+t)
  \in S$.
 \item for every $u,v \in {}^{<\omega}2$ of the same
   length, $(u,v,0^{|u|}) \in S$ implies $u=v$.
\end{enumerate}
\end{proposition}

A function $f \colon \seqo \to \seqo$ is said to be \emph{Lipschitz}
if $s \subseteq t \imp f(s) \subseteq f(t)$ and $|s| = |f(s)|$ for
each $s,t \in \seqo$. Consider now the space $\mathcal{T}$ of all
normal trees on $2 \times \omega$, and for $S,T \in \mathcal{T}$ put
$S \leq_{max} T$ if and only if there exists a Lipschitz function $f
\colon \seqo \to \seqo$ such that $(u,s) \in S \imp (u,f(s)) \in T$
for every $u \in {}^{<\omega}2$ and $s \in \seqo$ of the same length
(this in particular implies $p[S] \subseteq p[T]$).

\begin{theorem}[\cite{louros}]\label{theomax}
The quasi-order $\leq_{max}$ is complete for analytic quasi-orders.
\end{theorem}

Indeed, consider an arbitrary analytic
quasi-order $R$ on ${}^\omega 2$ together with any tree $S$ satisfying
Proposition \ref{propnormalform}; then define for every $x \in
{}^\omega 2$ the normal tree $S^x = \{(u,s) \mid (u,x \restriction
|u|,s) \in S \}$.
It follows that $x\, R\, y$ if and only if $S^x \leq_{max} S^y$ (for $x,y
\in {}^\omega 2$).
Moreover, as discussed in \cite{FriMot}, the map which sends $x$ to $S^x$ is continuous and
injective.
Given a quasi-order $R$ on ${}^\omega 2$ and $x \in
{}^\omega 2$, from
now on we will denote by $S^x$ the normal tree defined as above.

\section{Construction of sufficiently rigid trees}

This section contains an adaptation of a construction from
\cite{FriMot}, which will be needed in the sequel. The possibility of
such a construction was already noted in \cite{FriMot}, but we will
explicitly give definitions and related proofs here for the sake of
completeness.
The symbol $\sqsubseteq $ will denote embeddability between countable
combinatorial trees (that is, connected acyclic graphs).

Let $\# \colon \seqo \to \omega$ be
any bijection and $\theta \colon {}^{<\omega }2\to\omega $ be a
bijection such that $|u|<|v|$ implies $\theta (u)<\theta (v)$.

Let $G_0$ be the combinatorial tree obtained from $ \seqo $ by
splitting each edge linking $s \restriction (|s|-1)$ to $s$ by inserting a new
vertex $s^*$ in between, for $s \in \seqo \setminus \{ \emptyset \}$.
Next, define a combinatorial tree $G_1$ by adding to $G_0$ new
vertices $s^+$, $s^{++}$, $(s^{++},i,j)$, for $s \in \seqo$, $0 \leq
j \leq i \leq \# s + 2$, and then link
 $s$ to $s^+$, $s^+$ to $s^{++}$, $s^{++}$ to $(s^{++},i,0)$, and
 $(s^{++},i,j)$ to $(s^{++},i,j+1)$ whenever these vertices are
 defined (see Figure \ref{G1}).

\begin{figure}\usetikzlibrary{calc}
\begin{tikzpicture}[grow'=up]
\node {$s \restriction (|s|-1)$}
 child {node {$s^*$}
   child {node (s) {$s$}
     child {node {$s^+$}
       child {node {$s^{++}$} [sibling distance=2cm]
         child {node {$(s^{++},0,0)$}}
         child {node {$(s^{++},1,0)$}
           child {node {$(s^{++},1,1)$}}}
         child {node {$(s^{++},2,0)$}
           child {node {$(s^{++},2,1)$}
             child {node {$(s^{++},2,2)$}}}}
         child {node {$(s^{++},3,0)$}
           child {node {$(s^{++},3,1)$}
             child {node {$(s^{++},3,2)$}
               child {node {$(s^{++},3,3)$}}}}}}}
     child[missing]
     child[missing]
     child {node {$(s^\smallfrown \langle 0 \rangle)^*$}
       child {node (dd) {$s^\smallfrown \langle 0 \rangle$}}}
     child {node (last*) {$(s^\smallfrown \langle 1 \rangle)^*$}
       child {node (last) {$s^\smallfrown \langle 1 \rangle$}}}
}};
\node (xx) at ($ (last) + (1.5,0)$) {$\dots$};
\node (xx*) at ($ (last*) + (1.5,0)$) {$\dots$};
\node () at ($ (last) + (0,1)$) {$\vdots$};
\node () at ($ (dd) + (0,1)$) {$\vdots$};
\end{tikzpicture}
\caption{\label{G1}A portion of $G_1$ around the unique $s$ such that $\# s =1$.}
\end{figure}

Given a normal tree $T$ on $2 \times \omega$, we define a
combinatorial tree $G_T$ starting from $G_1$
as follows: for every $u \in {}^{< \omega}2$ and $s \in \seqo$
such that $(u,s) \in T$, add to $G_1$ vertices $(u,s,x)$, where $x$ is an
initial subsequence of
$0^{2 \theta(u)+4}$ or $x = 0^{2 \theta(u)+2} {}^\smallfrown 1$, and
connect $(u,s,\emptyset)$ to $s$ and $(u,s,x)$ to  $(u,s,x')$ just in
case one of $x,x'$ is an immediate successor of the other. Now it is
easy to see how to reprove \cite[Theorem 3.1]{louros}, i.e.\ that $S
\leq_{max} T  \iff G_S \sqsubseteq G_T$. For one
direction, assume first that $S,T$ are normal trees on $2 \times
\omega$ such that
$S \leq_{max} T$: we claim that this can be witnessed by an injective Lipschitz
function $f
\colon \seqo \to \seqo$ such that $\forall s\in \seqo (\#s \leq\# f(s))$. To
see this, let $f' \colon \seqo \to \seqo$ be any Lipschitz function
witnessing $S \leq_{max} T$, and define inductively $f(s)$ to be the
$\leq_{lex}$-least $t$ such that:

\begin{itemize}
\item $t$ extends $f(s\restriction (|s|-1))$ if $|s|>0$,
\item $f'(s) \leq t$,
\item $f(s')\neq t$ for every $s' \in {}^{|s|} \omega$ such that $s'<_{lex} s$,
\item $\# s \leq \# t$.
\end{itemize}

If $f$ is as above,
then define the embedding $g$ from $G_S$ to $G_T$ by sending $s$ to
$f(s)$, $s^*$ to $f(s)^*$, $s^+$ to $f(s)^+$, $s^{++}$ to $f(s)^{++}$,
$(s^{++},i,j)$ to $(f(s)^{++}, i,j)$, and $(u,s,x)$ to
$(u,f(s),x)$. For the other direction, assume $S \neq \emptyset$.
Notice that all the points in
$G_S$ have valence $\leq 2$ except for those of the form $s \in
\seqo$ (which have valence $\omega$), $s^{++}$ (which have valence
$\#s +4$), and $(u,s,0^{2 \theta(u)+2})$ (which have valence
$3$). Moreover, the distance from $s$ to $s^{++}$ is always $2$
 and vertices of the form $(u,s,0^{2 \theta(u)+2})$
have odd distance (greater than $2$) from
vertices in $\seqo$, this distance being determined by
function $\theta $, which is increasing with respect to the length of
the argument.
So if $g$ is an embedding
of $G_S$ in $G_T$, it follows that
$g(\emptyset , \emptyset , 0^{2\theta(\emptyset)+2})  =  (\emptyset ,
\emptyset, 0^{2\theta(\emptyset)+2})$, and hence $g(\emptyset)  =  \emptyset$.
Now, by induction, conclude that
$f = g \restriction \seqo$ is such that $f(\emptyset) = \emptyset$,
${\rm range}(f) \subseteq \seqo$, and $f$
witnesses $S \leq_{max} T$, similarly as in the proof of \cite[Theorem
3.1]{louros}.

\begin{lemma}\label{lemma1}
 Let $S,T$ be normal trees on $2\times\omega $, and $G_S$ and $G_T$ be defined as before.
If $S \neq T$ then $G_S \not\cong G_T$.
\end{lemma}

\begin{proof}
 Assume that $i$ is an isomorphism between $G_S$ and
$G_T$, in order to show $S\subseteq T$ (and symmetrically $T\subseteq S$). Since $i$ preserves distances and valences, each $s^{++}$
must be mapped to itself, as it is the unique point in both
graphs with valence $\# s +4$. Therefore also $s$ is mapped to
itself, as it is the unique point (in both graphs) which is at
distance $2$ from $s^{++}$ and has valence $\omega$. Therefore $i
\restriction \seqo$ is the identity. Suppose now $(u,s) \in S$:
 as in the proof of \cite[Theorem 3.1]{louros}, the point
 $(u,s,0^{2\theta(u)+2})$  must be sent by $i$ to $(u,i(s),0^{2\theta(u)+2}) =
 (u,s,0^{2 \theta(u)+2})$, which means that $(u,s) \in T$. Hence $S
 \subseteq T$.
\end{proof}

From now on $ \L $ will be the language with a single binary relation
symbol. Then each $G_T$ defined as above will be viewed as an element
of $Mod_{ \L }$ (it is easy to see that each $G_T$ can be coded
Borel-in-$T$ as a graph on $\omega$). We shall denote by $\G$ the set
of all such $G_T$. The following corollary gathers two important
consequences of Lemma \ref{lemma1}.
\begin{corollary}\label{mathbbg}
\begin{enumerate}[1)]
\item $\G $ is a Borel subset of $Mod_{ \L }$, so it is a standard Borel space.
\item On $\G$ equality and isomorphism coincide.
\end{enumerate}
\end{corollary}

The restrictions to $ \G $ of binary relations defined on combinatorial trees (like equality or embeddability) will often be denoted with a subscript $ \G $.

\begin{lemma}\label{lemma2}
 For every distinct $p,q \in S_\infty$ and every  normal
 tree $T$ on $2 \times
 \omega$, we have $j_\L(p,G_T) \neq j_\L(q,G_T)$.
\end{lemma}

\begin{proof}
 Assume that $j_\L(p,G_T) =
j_\L(q,G_T)$: we want to show that in this case $p=q$. To see this first
check that if $(g_0, \dotsc ,g_n)$ is a path in $G_T$,
and $p(g_0) = q(g_0)$ and $p(g_n) = q(g_n)$,
then $p(g_k) = q(g_k)$ for every $k \leq n$ (using the acyclicity of
$G_T$). Therefore, in our case it will be enough to show that
$p$ and $q$ coincide on vertices of the form $\emptyset$,
$(s^{++},i,i)$, $(u,s,0^{2 \theta(u)+4})$  and $(u,s,0^{2 \theta(u)+2}
{}^\smallfrown 1)$, and this amounts to showing that each of these points is
the unique element of $G_T$ which satisfies a certain  property that
can be expressed in terms of valence and distance. First note that $p$
and $q$ must coincide on vertices of the form $s^{++}$, as these are
the unique points with valence $\# s+4$: therefore we get that $p$ and
$q$ must coincide also on elements of the form $s \in \seqo$ as $s$ is
the unique point of $G_T$ which has valence $\omega$ and is at
distance $2$ from $s^{++}$ (in particular we have $p(\emptyset) =
q(\emptyset)$). Then $p$ and $q$ must coincide on elements of the form
$(s^{++},i,i)$, as these are the unique points of $G_T$ with
valence $1$, distance $i+1$ from $s^{++}$ and distance $i+3$ from $s$,
and on elements of the form $(u,s,0^{2 \theta(u)+2})$ because these
are the unique points with valence $3$, distance $2\theta(u)+3$ from
$s$, and greater distance from any $s' \in \seqo$ distinct from
$s$. Finally, $p$ and $q$ must coincide on elements of the form
$(u,s,0^{2 \theta(u)+4})$ or $(u,s,0^{2 \theta(u)+2} {}^\smallfrown
1)$, as these are the unique points with valence $1$ which are at
distance $2$ and $1$, respectively, from $(u,s,0^{2 \theta(u)+2})$.
\end{proof}

\begin{corollary}\label{cornontrivial}
 Let $T$ be a normal tree on $2 \times \omega$. Then $G_T$ is rigid (i.e. its unique automorphism is the identity).
\end{corollary}

\section{The main theorem}\label{sectionsaturation}

Theorem \ref{theorsaturation} below constitutes our main tool for
proving invariant universality of analytic quasi-orders.
To establish it, we have to deal with the problem of $E$-saturating in
a Borel way
the range of a reduction between embeddability on $\G$ and an
analytic quasi-order $S$ with $E$ an equivalence subrelation of
$E_S$.

First we need the following
technical lemma, which is essentially a reformulation of
\cite[Theorem 1.2.4]{beckerkechris}.
Let $Y$ be a Polish group and $Z$ a closed subgroup of $Y$.
By a theorem of Burgess (see \cite[Theorem 12.17]{Kechris1995}),
there is a Borel selector
$s\colon Y \to Y$ for the equivalence relation on $Y$ whose
classes are the (left) cosets of
$Z$, and, consequently, a Borel transversal $T= \{ y \in Y \mid
s(y)=y \} = {\rm range}(s)$ for the same
equivalence relation.
Next lemma is a parametrized version of Burgess' theorem.

\begin{lemma}\label{lemmaBurgess}
Let $X$ be a standard Borel space, $Y$ a Polish group and $\Sigma
\colon X \to G(Y)$ a Borel map. Then there is a Borel function $s \colon X \times
Y \to Y$ such that $s_x \colon Y \to Y  \colon y \mapsto s(x,y)$ is a
selector for the equivalence relation $E_x$ whose classes are the (left) cosets of $\Sigma(x)$. Therefore $T = \{ (x,y) \in X \times Y \mid s_x(y)
= y \}$ is Borel as well, and $T_x = \{ y \in Y \mid (x,y) \in T \}$ is
a Borel transversal for $E_x$.
\end{lemma}

\begin{proof}
Define $s = d \circ f_1 \circ f_0$ where:
\begin{itemize}
\item $f_0 \colon X\times Y \to G(Y)\times Y$ is defined by
  $f_0(x,y)=(\Sigma(x),y)$;
\item $f_1 \colon G(Y)\times Y \to F(Y)$ is defined by $f_1(G,y)=yG$;
\item $d \colon F(Y)\to Y$ is a Borel function such that $\forall F
  \in F(Y) \setminus \{\emptyset\}\ (d(F)\in F)$.
\end{itemize}

  It is clear that each $s_x$ is a selector for $E_x$. To see that $s$ is Borel, simply notice that $f_0$ is Borel because $\Sigma$ is such, while $f_1$ is Borel because if
$\langle d_n \mid n \in \omega \rangle$ is any sequence of Borel
functions from $F(Y)$ into $Y$ such that $\langle d_n (F) \mid n \in
\omega \rangle$ is a dense subset of every $F \in F(Y) \setminus \{ \emptyset \}$, then
\[
yG=F \iff \forall n \in \omega (y d_n(G) \in F \land y^{-1} d_n (F) \in G). \qedhere
\]
\end{proof}

\begin{theorem}\label{theorsaturation}
Let $S$ be an analytic quasi-order on a standard Borel space $Z$, and
$E \subseteq E_S$  be an analytic equivalence relation on the same
space. Suppose there exists a Borel function $f \colon \G \to Z$ which
simultaneously witnesses ${\sqsubseteq_\G} \leq_B S$ and ${=_\G} \leq_B
E$ (which is the same as ${\cong_\G} \leq_B E$). Furthermore, let $Y$ be a
Polish group, $a$ a Borel action of $Y$ on a standard Borel space $W$,
and $g \colon Z \to W$ witness $E \leq_B E_a$.

Consider the map $\Sigma \colon \G \to G(Y)$ which assigns to $G \in \G$ the
stabilizer of $(g \circ f)(G)$ with respect to $a$, i.e.\
\[
\Sigma(G) = \{y \in Y \mid a(y, (g \circ f)(G)) = (g \circ f)(G)\}.
\]

If $\Sigma$ is Borel, then the pair $(S,E)$ is invariantly universal.
\end{theorem}

\begin{proof}
We start by showing that it is enough to prove the following claim.
\begin{claim}\label{claimmain}
For any Borel $B \subseteq \G$, the $E$-saturation ${\rm Sat}(f(B))$
of $f(B)$ is Borel.
\end{claim}

Granting this claim, the proof of the theorem can be completed as follows:
Let $R$ be an arbitrary analytic quasi-order on ${}^\omega 2$, and let $S^x$ be normal trees defined as
before, so that $x\, R\, y \iff {S^x \leq_{max} S^y} \iff {G_{S^x}
\sqsubseteq_{ \G }G_{S^y}}$ and
the map $x \mapsto G_{S^x}$ is injective. Now consider $B_R = \{ G_{S^x}
\mid x \in {}^\omega 2 \}$: being a Borel subset of $\G$, from the claim we get that $C = {\rm Sat}(f(B_R))$ is Borel and $E$-invariant. The map ${}^{\omega
}2\to Z \colon x \mapsto f(G_{S^x})$ is clearly a reduction of $R$ to
$S \restriction C$. For the other direction, note that $f(B_R)$
consists of $E$-incomparable elements (this is by our assumption that
$f$ reduces equality on $\G$ to $E$), so that the map which sends
$y \in C$ to the unique $x \in {}^\omega 2$ such that $f(G_{S^x})\, E\, y$
is a well-defined reduction of $S \restriction C$ to $R$, and it is Borel as its graph is analytic.

\smallskip
It remains to prove the claim. Since ${\rm Sat}(f(B))$ is
easily seen to be analytic, it is enough to show that it is also a
co-analytic set. Apply Lemma
\ref{lemmaBurgess} to the map $\Sigma\restriction B$, and let
$T$ be the resulting Borel subset of $B \times Y$. Since
\[
P = \{ z \in Z \mid \exists \, ! (G,y) \in B\times Y ({(G,y) \in T}
\wedge {a(y, g(f(G))) = g(z)} ) \}
\]
is the set of uniqueness of a Borel set, and hence co-analytic by a
classical result of Luzin (see e.g.\ \cite[Theorem
18.11]{Kechris1995}), it is enough to show that ${\rm Sat}(f(B)) = P$.
One inclusion is obvious, as if $g(z) = a(y, g(f(G)))$ for some $G \in
B$ and $y \in Y$, then $z\, E\, f(G)$ and hence $z \in {\rm
Sat}(f(B))$. For the other direction, assume that $G \in B$ is such
that $z\, E\, f(G)$. Since $g$ reduces $E$ to $E_a$, there is $\bar{y}
\in Y$ such that $a(\bar{y}, g(f(G)))= g(z)$. Let $y \in Y$ be in the
same (left) coset of $\Sigma(G)$ of $\bar{y}$ and such that $(G,y) \in
T$ (such a $y$ must exist because the vertical section $T_G$ meets all
left cosets of $\Sigma(G)$): since $\bar{y}^{-1} y \in \Sigma(G)$ we
have that $a(\bar{y}^{-1} y, g(f(G))) = g(f(G))$, whence $a(y,g(f(G)))
= a(\bar{y},g(f(G))) = g(z)$. So $(G,y) \in B \times Y$ is such that
$(G,y) \in T$ and $a(y,g(f(G))) = g(z)$: we want to prove that $(G,y)$
is also the unique pair satisfying these conditions, so that $z \in P$.
Assume
 that $(G',y') \in T$ is such that
$a(y', g(f(G'))) = g(z)$. The last condition implies $f(G') \, E\, z\, E\,
f(G)$, so that $G=G'$. But
then $a(y',g(f(G))) = g(z) =
a(y,g(f(G)))$ implies that $y$ and $y'$ are in the same (left) coset of
$\Sigma(G)$: since $(G,y),(G,y') \in T$ and $T_G$ is a
transversal for the equivalence relation on $Y$ whose classes are the
(left) cosets of $\Sigma(G)$, we get also that $y = y'$, so that
$(G',y') = (G,y)$.
\end{proof}

\begin{remark}
Notice that our proof of Theorem \ref{theorsaturation} actually shows that, given $S,Z$ and $f$ as in the hypotheses of such theorem, Claim \ref{claimmain} already implies that the pair $(S,E)$ is invariantly universal, and hence such condition could also be useful to deal with the case when $E$ is not Borel reducible to an orbit equivalence relation (see e.g.\ the observation after Question \ref{questionfinal} below).

Moreover,
as kindly pointed out by the anonymous referee, given $S,Z,f,E_a$ and $g$ as in the hypotheses of Theorem \ref{theorsaturation}, $\Sigma$ is a Borel function if and only if the $E_a$-saturation of $g(f(\G))$ is Borel\footnote{One direction of the equivalence can be obtained using an argument similar to the proof of Claim \ref{claimmain} given above. For the other direction, since $g \circ f$ is injective on $\G$ we get that $g(f(\G))$ is a Borel transversal for the restriction of $E_a$ to its $E_a$-saturation (which is a Borel set by hypothesis). This implies that such restriction is a smooth orbit equivalence relation and hence, in particular, Borel. But then the map $\sigma$ assigning to each $w$ in the $E_a$-saturation of $g(f(\G))$ its stabilizer is Borel, and hence $\Sigma$ is a Borel map since $\Sigma = \sigma \circ g \circ f$.}.
As demonstrated by the applications in Section \ref{sectionapplications},  despite the fact that the last requirement could seem easier to check, the only way we found to achieve this goal was exactly the content of Theorem \ref{theorsaturation}, i.e.\ to prove that the map $\Sigma$ is Borel.
\end{remark}

To apply Theorem \ref{theorsaturation} we have thus to perform three
steps:
\begin{enumerate}[\quad(a)]
\item Find $f \colon \G \to Z$ and check that it witnesses both
    ${\sqsubseteq_\G} \leq_B S$ and ${=_\G} \leq_B E$. In all
    applications, $S$ is already known to be a complete analytic quasi-order
    by a proof showing ${\sqsubseteq_\G} \leq_B S$. Thus $f$ will be the
    function defined in that proof (or a minor modification of it),
    and this step consists in
    checking that such $f$ witnesses also ${=_\G} \leq_B E$, usually
    performing a detailed analysis of $f$.
\item Find the appropriate $Y$, $W$, $a$ and $g$. Often $E$ itself
    is an orbit equivalence relation of a Polish group action, so
    that $W=Z$, $Y$ and $a$ are given, and we tacitly assume
    that $g$ is the identity.
\item Prove that $\Sigma$ is Borel. In most applications
    this is the most delicate part of the proof.
\end{enumerate}

\section{Applications}\label{sectionapplications}

In this section we apply Theorem \ref{theorsaturation} to show that
all analytic quasi-orders we know to be complete are in fact invariantly universal when coupled with
natural equivalence relations. In
particular, we
answer affirmatively some questions
posed in \cite[Questions 6
and 8]{FriMot} (although in the first case we have just a partial
answer, as the case of epimorphisms preserving relations and functions
in both directions still remains completely open).
These asked whether the quasi-orders induced by epimorphisms between countable structures, isometric embeddings between ultrametric Polish spaces, continuous embeddings between compact metrizable spaces and linear isometric embeddings between separable Banach spaces are invariantly universal when paired with, respectively, isomorphism, isometry, homeomorphism and linear isometry.

The section is quite long and contains a great variety of applications
in many areas of mathematics, most of them involving some specific new
idea or technique. Subsection \ref{ref1} deals with the relation of
being epimorphic image between graphs, and constitutes the first
nontrivial application of Theorem \ref{theorsaturation}. Subsection
\ref{coloredlinearorders} gives some natural  examples of invariantly
universal quasi-orders in the realm of combinatorics. Subsection
\ref{dop} uses the results from the preceding subsection to give other
combinatorial examples which are in turn used in Subsection
\ref{dendrites}, where we study invariant universality in topology
(this is also the first application of Theorem \ref{theorsaturation}
where the analytic equivalence relation under consideration is not an
orbit equivalence relation itself, but is only reducible to such a
relation). Subsection \ref{metric} provides applications in metric
space theory. Subsection \ref{banach} deals with the case of separable
Banach spaces, and differs from all other subsections in that we need
to build a new and more manageable reduction of linear isometry
(restricted to some suitable and abstractly defined class of separable
Banach spaces) to an orbit equivalence relation (namely, to an
isomorphism relation). Finally, Subsection \ref{further} contains
applications of our main result to all remaining examples of complete
analytic quasi-orders (in this case we sketch the proofs of the results
as, contrarily to all other subsections, they do not involve any
genuinely new idea, but are combinations of techniques which already
appeared in previous applications).

The following notation, originating from \cite{frista1989}, will be
used several times. For each $n \in \omega $, let $TY_n$ be the set of
quantifier free types for the first $n$ variables in $\L$ (an empty
$0$-type is also considered here) and let $TY = \bigcup_{n\in\omega}
TY_n$. Fix a bijection $e \colon \omega \to TY$ such that $e(i) \in
TY_n, e(j) \in TY_m, n<m$ imply $i<j$. (This is possible as $ \mathcal
L $ consists of just one relation symbol.) Each $i \in \omega$ is the
\emph{code} of type $e(i)$. For $G$ an $\L$-structure on $\omega$ and
$t \in \seqo$, let $\tau_G (t) \in \omega$ be the code of the
quantifier free type of $t$ in $G$.

\subsection{Epimorphisms between graphs}\label{ref1}
Given graphs $H,H'$ on $\omega $, define the relation $\preceq_{epi}$ of \emph{being
epimorphic image} by letting $H\preceq_{epi}H'$ if and only if there is a surjection $\gamma \colon \omega\to\omega $ such that
\[ n,m \mbox{ are adjacent in } H'\Rightarrow \gamma (n),\gamma (m) \mbox{ are adjacent in } H. \]
We recall here the construction of \cite[Theorem 1]{Camerlo2005},
which defines a continuous function $G\mapsto G^*$ from the class of
graphs on $\omega $ to itself reducing the relation of embedding to the relation $\preceq_{epi}$.
We will prove some further properties of this construction, which will
be used to show invariant universality of $\preceq_{epi}$.

Let $\{ N_t  \mid t\in \seqo \}$ be a partition of $\omega $ into
infinite sets.
Within each $N_t$ fix distinct elements $a^t,c_i^t$, for $i \in \omega$
so that $N_t \setminus \{ a^t,c_i^t \mid i\in\omega \}$ is still infinite.
For any $t \in \seqo , n\in\omega $ let $L_{tn}$ be a graph on $N_t$ whose adjacency relation is defined according to the following clauses:
\begin{enumerate}
\item $N_t$ is the disjoint union of two sets $B_{tn}= \{ a^t,b_1^{tn},
  \dotsc ,b_{n+2}^{tn}\},C_{tn}=\{ c_i^t,d_i^{tn}\mid i\in\omega\} $ (where the elements in each list are pairwise distinct);
\item each of $B_{tn},C_{tn}$ forms a clique in $L_{tn}$;
\item in addition to the other elements of $B_{tn}$, vertex
  $b_{n+2}^{tn}$ is adjacent to all elements of $C_{tn}$.
\end{enumerate}
Given a graph $G$ on $\omega $, the adjacency relation on the graph $G^*$ is defined as follows:
\begin{itemize}
\item the adjacency relation on each $N_t$ is given by $L_{t\tau_G(t)}$;
\item for each $t\in \seqo , i\in\omega $, vertices $c_i^t,a^{t^{\smallfrown }i}$ are adjacent.
\end{itemize}

\begin{lemma}\label{lemmaiso}
$G \cong H \iff G^* \cong H^*$.
\end{lemma}

\begin{proof}
Let $ \fhi \colon \omega \to \omega$ be an isomorphism between $G$ and $H$.
Then $\fhi$ induces componentwise a bijection $\fhi' \colon \seqo \to
\seqo$ with the property $\forall t \in \seqo (\tau_G(t)=\tau_H
\fhi'(t))$.
An isomorphism $\psi \colon \omega \to \omega$ between $G^*$ and $H^*$
is obtained by setting for every $t \in \seqo$:
\begin{itemize}
\item $\psi (a^t)=a^{\fhi'(t)}$;
\item $\psi (b_j^{t\tau_G(t)})=b_j^{ \fhi'(t)\tau_H \fhi'(t)}$, for $j
  \in \{ 1, \dotsc ,\tau_G(t)+2 \}$;
\item $\psi (c_i^t)=c_{\fhi(i)}^{\fhi'(t)}$,
  $\psi(d_i^{t\tau_G(t)})=d_i^{\fhi'(t)\tau_H \fhi'(t)}$, for $i\in
  \omega$.
\end{itemize}

Conversely, let $\psi \colon \omega \to \omega$ be an isomorphism
between $G^*$ and $H^*$.
Observe that $\psi(a^t)=a^{t'}$, for $t\neq\emptyset $, and the
map $\psi'$ sending $t$ into
$t'$ and $\emptyset$ into itself  is a Lipschitz bijection $\seqo \to \seqo$ such that
$\forall t \in \seqo (\tau_G(t)=\tau_H\psi'(t))$.
Imitating the proof of \cite{Camerlo2005}, by a back and forth argument inductively construct $u \in S_{\infty}$
such that $v = \bigcup_{n \in \omega} \psi'(u \restriction n) \in
S_\infty$, $\forall n \in \omega (\tau_G(u \restriction n)=\tau_H
\psi'(u \restriction n))$.
Then the bijection $u(n) \mapsto v(n)$ is an isomorphism between $G$ and $H$.
\end{proof}

\begin{defin}
Let $G$ be a graph.
An automorphism $\fhi$ of $G$ is \emph{simple} if whenever $ \fhi (u)
\neq u$, then $u$ belongs to a unique maximal clique of $G$ and
$\fhi(u)$ belongs to this same clique.
\end{defin}

So for graphs of the form $G^*$, an automorphism is simple if and only
if it leaves  all $a^t$ fixed, $t\neq\emptyset $.
We remark that such an automorphism fixes the elements of the form
$b_{\tau_G(t)+2}^{t\tau_G(t)}$ and $c_i^t$ too,  while it can
permutate sets of the forms $\{ b_1^{t\tau_G(t)},\dotsc
,b_{\tau_G(t)+1}^{t\tau_G(t)}\} $ and $\{ d_i^{t\tau_G(t)} \mid i\in\omega \}$.

\begin{lemma}\label{lemmasimple}
If the only automorphism of $G$ is the identity then all automorphisms
of $G^*$ are simple.
\end{lemma}

\begin{proof}
Suppose $\psi \colon \omega \to \omega $ be a non-simple automorphism of $G^*$.
Then there exist distinct elements $t, s \in {}^m\omega$, for some $m>0$, such that $\psi (a^t)=a^s$.
Notice that if $t(n_0),\dotsc ,t(n_k)$ are distinct and include all
values taken by $t$, then $s(n_0),\dotsc ,s(n_k)$ are distinct and
include all values of $s$.
By a back and forth argument, build $u,v\in {}^{\omega }\omega$ such
that $(t(n_0),\dotsc ,t(n_k))^{\smallfrown }u,(s(n_0),\dotsc
,s(n_k))^{\smallfrown }v \in S_{\infty }$ and, letting
$x=t^{\smallfrown }u,y=s^{\smallfrown }v$, the relation $\forall n>0(\psi(a^{x\restriction n})=a^{y\restriction n})$ holds.
Since $\forall n \in \omega (\tau_G(x\restriction
n)=\tau_G(y\restriction n))$, the function $x(n)\mapsto y(n)$ is a
non-trivial automorphism of $G$.
\end{proof}

\begin{lemma}
Given a graph $G$ on $\omega $, the set $H_G$ of simple
automorphisms of $G$ is a closed subgroup of $S_\infty$.
\end{lemma}

\begin{proof}
The fact that $H_G$ is a group uses the fact the if
$u \in \omega$ belongs to a unique maximal clique in $G$, then the same
holds for $\fhi (u)$ for any automorphism $\fhi$ of $G$.

It is now enough to show that $H_G$ is closed in the automorphism
group of $G$, since the latter is closed in $S_\infty$.
Let $\fhi_n, \fhi$ be automorphisms of $G$ such that $\lim_{n \to
  \infty} \fhi_n= \fhi \notin H_G$.
Let $u \in \omega $ be such that $\fhi(u) \neq u$ where either $u$
belongs to more than one maximal clique of $G$ or $\fhi(u)$ does not
belong to the same maximal clique as $u$.
Then eventually $\fhi_n(u)= \fhi(u)$, so $\fhi_n \notin H_G$.
\end{proof}

\begin{lemma}\label{lemmaborelmap}
The map $G \mapsto H_G$ from the space of graphs on $\omega $ to
$G(S_\infty)$, assigning to each graph the group of its simple
automorphisms, is Borel.
\end{lemma}

\begin{proof}
It in enough to show that given any finite injective sequence
$(a_0,\dotsc ,a_n)$ of natural numbers, the class of graphs having
a simple automorphism extending $(a_0,\dotsc ,a_n)$ is Borel.
A graph $G$ belongs to this class if and only if it satisfies the following Borel
conditions, for all $i,j \in \{ 0,\dotsc ,n\} $:
\begin{enumerate}
\item $i,j$ are adjacent in $G$ if and only if $a_i,a_j$ are adjacent in $G$;
\item if $i \neq a_i$ then:
\begin{itemize}
\item for all distinct $u,v\in\omega $, if $i$ is adjacent in $G$ both
  to $u,v$, then $u,v$ are adjacent in $G$;
\item for all distinct $u,v\in\omega $, if $a_i$ is adjacent in $G$ to
  both $u,v$, then $u,v$ are adjacent in $G$;
\item $i,a_i$ are adjacent in $G$. \qedhere
\end{itemize}
\end{enumerate}
\end{proof}

\begin{theorem}
The relation $\preceq_{epi}$ of being epimorphic image on countable
graphs is invariantly universal. In view of Lopez-Escobar's theorem, this means that if $R$ is an analytic quasi-order on a
standard Borel space $X$ then
 there is an $\L_{\omega_1\omega}$-sentence $\fhi$
   such that $R$ is
Borel bireducible with the relation $\preceq_{epi}$ on $Mod_\fhi$, where
each element of $Mod_\fhi$ is a graph.
\end{theorem}

\begin{proof}
The $\preceq_{epi}$-isomorphism relation is
isomorphism, which is induced by the logic action of
$S_\infty$ on $Mod_\L$. By Lemma \ref{lemmaiso}, the map $G \mapsto
G^*$ reducing embeddability to $\preceq_{epi}$ reduces also isomorphism
to isomorphism. Moreover, by Corollary \ref{cornontrivial} and Lemma
\ref{lemmasimple} we get that for   $G \in \G$ the group
of automorphisms of $G^*$ coincides with the group $H_G$ of its
simple automorphisms, and then by Lemma \ref{lemmaborelmap} the map
$\Sigma$ assigning to each $G \in \G$ the group of automorphisms of $G^*$ is
Borel. Therefore we can apply Theorem \ref{theorsaturation} to get
that $(\preceq_{epi}, \cong)$ is invariantly universal.
\end{proof}

\subsection{Embeddings between colored linear orders}\label{coloredlinearorders}

This section and the next deal with various kinds of embeddings for linear orderings.

Denote by $LO$ the Polish space of (strict) linear orderings on $\omega$ and
let $R$ be any quasi-order on $\omega$. Define the analytic quasi-order
$\preceq_R$ on $LO \times {}^\omega \omega$ by letting $(L,c)
\preceq_R (L',c')$ if and only if there is an
embedding\footnote{We can also replace ``embedding'' with
  ``homomorphism'' or ``weak-homomorphism'', as all these notions
  coincide on strict linear orderings.} $g$ of $L$
into $L'$ which preserves $R$, i.e.\ such that $c(n)\, R\, c'(g(n))$ for
every $n \in \omega$. In \cite{MarRos} it was proved that $\preceq_=$
is a complete analytic quasi-order, and this
result was then extended in \cite{Camerlo} using
the following construction.

Fix an enumeration $\langle k_n  \mid n \in \omega \rangle$ of $\omega$
such that every natural number is listed infinitely many times. In this
way, each element $t \in \seqo$ can be seen as a label for the sequence
$\lambda_t = \langle k_{t(0)}, k_{t(1)}, \dotsc, k_{t(|t|-1)} \rangle
\in \seqo$. Given any graph $G$, define the colored order $L_G = (L_G,
c_G)$ by replacing in the lexicographic order of $\seqo \setminus \{
\emptyset \}$ each $t$ with an interval $I^t_G$ (later called a
\emph{block}) of order type $\omega^{2 \tau_G(\lambda_t)}$ (this block
and its elements will be said to \emph{correspond} to $t$ or to
\emph{replace} $t$), all these points colored by $\tau_G(\lambda_t)$.
In \cite{Camerlo}, it was shown that for every $G,G'$ if $G$ embeds
into $G'$ then $L_G \preceq_= L_{G'}$, and if $L_G \preceq_\ge L_{G'}$
then $G$ embeds into $G'$ (so that, in particular, any analytic
quasi-order $S$ on $LO \times {}^\omega \omega$ such that ${\preceq_=}
\subseteq S \subseteq {\preceq_{\geq}}$ is complete analytic).

Note that each quasi-order of the form $\preceq_R$ is a morphism
relation, and both $\preceq_=$ and $\preceq_\ge$ have the same
associated isomorphism relation $\cong_=$ of color
preserving isomorphism. We will now simultaneously show that both
$\preceq_=$ and
$\preceq_{\geq}$ are in fact invariantly universal.

\begin{lemma}\label{lemmaidentity}
If $L_G \cong_= L_{G'}$ via the isomorphism $g \colon \omega \to \omega$,
then $G = G'$ and $g = id$.
\end{lemma}

\begin{proof}
Since $g$ must be color preserving and blocks corresponding to sequences
of different length are necessarily colored with different colors,
we get that the restriction $g_n$ of $g$ to any suborder of $L_G$ with
domain $L^n_G = \bigcup \{ I^t_G \mid t \in {}^n \omega \}$ must be an
isomorphism between $L^n_G$ and $L^n_{G'}$. But any $L^n_G$ is now a
well-order, therefore one inductively gets $\tau_G(\lambda_t) =
\tau_{G'}(\lambda_t)$ for any $t \in {}^n \omega$, which implies that
$G = G'$ and hence $L_G = L_{G'}$. Moreover,
each $g_n$ being defined on a well-order, it must be the identity,
hence $g = id$.
\end{proof}

\begin{theorem}\label{theorcoloredorders}
If $S$ is any analytic quasi-order on the space of colored linear orderings on $\omega $
such that ${\preceq_=}\subseteq S \subseteq {\preceq_\geq}$, then $(S, \cong_=)$
is invariantly universal. In particular, $\preceq_=, \preceq_\ge$, as
well as any quasi-order of the form $\preceq_R$ for $R
\subseteq {\ge}$, are invariantly universal.
\end{theorem}

\begin{proof}
First observe that $\cong_=$ is induced by an obvious action of
$S_\infty$ on $LO \times {}^\omega \omega$ sending $(p,(L,c)) \in
S_\infty \times (LO \times {}^\omega \omega)$ into
$(L',c')$, where $L' = j_\L(p,L)$ and $c'(n) = c(p^{-1}(n))$. By Lemma
\ref{lemmaidentity}, the map $G \mapsto L_G$ reduces  $=_\G$ to $\cong_=$.
Finally, since
any color preserving automorphism of an $L_G$ must be the identity by Lemma
\ref{lemmaidentity} again, the function $\Sigma$ assigning to each $G
\in \G$ the
stabilizer of $L_G$ is constantly equal to $\{ id \} \in
G(S_\infty)$, hence Borel. Applying now Theorem \ref{theorsaturation}
we get the result.
\end{proof}

\begin{corollary}
If $R'$ is any quasi-order on $\omega$ containing an infinite descending sequence or an infinite anti-chain (i.e. is not a wqo) then $(\preceq_{R'}, \cong_=)$ is invariantly
universal.
\end{corollary}

Now we want to prove that the relation $\preceq_=$, when restricted to
colored orderings whose support is fixed to
be the order $\QQ$, is still
invariantly universal (such relation will be denoted by $\preceq_\QQ$). To begin with, we slightly modify a
construction from \cite[Lemma 3]{Camerlo}. Let $LO^*$ be the space of linear
orderings on $\omega \times \QQ$. To each $L = (L,c) \in LO \times
{}^\omega \omega$, associate $L^* = (L^*,c^*) \in LO^* \times
{}^{\omega \times \QQ} \omega$ in the following way:
\begin{itemize}
\item $L^*$ is the lexicographic product of $L$ and $\leq $ (so $L^*$ is a countable dense linear ordering without
endpoints);
\item $c^*(n,0) = 2c(n)+1$, and $c^*(n,q)
= 2 \eta(q)$ if $q \neq 0$, where $\eta$ is any bijection between $\QQ
\setminus \{ 0 \}$ and $\omega$.
\end{itemize}
As each colored ordering of the type
$L^*$ can
be Borel-in-$L$ identified with a coloring on
$\QQ$, we can identify $L^*$ with its coded version on $\QQ$. Now it
is easy to check that
the Borel map $L \mapsto L^*$ is a reduction of $\preceq_=$ on $LO
\times {}^\omega \omega$ to $\preceq_\QQ$. For one direction, if $g
\colon \omega \to \omega$ is an order and color preserving embedding
between $L$ and $L'$, then the map $(n,q) \mapsto (g(n),q)$ is a
witness of $L^* \preceq_\QQ (L')^*$. Conversely, any witness $h$ of $L^*
\preceq_{ \QQ }(L')^*$ must be such that $h(n,0) = (n',0)$ for some $n' \in
\omega$, and this induces an order and color preserving embedding $n\mapsto n'$
between $L$ and $L'$.

\begin{theorem}\label{theorQuniversal}
The quasi-order $\preceq_\QQ$ on ${}^\QQ \omega$ is invariantly universal.
\end{theorem}

\begin{proof}
The $\preceq_{ \QQ }$-isomorphism relation $\cong_\QQ$ is induced by
the obvious action of the group $Aut(\QQ)$ of automorphisms of $\QQ$.
Now consider the map $\G \to {}^\QQ \omega \colon G \mapsto L_G^*$,
defined as above. We first want to show that this is a reduction of
$=_{ \G }$ to $\cong_\QQ$. For the nontrivial direction, if $h$
witnesses $L_G^* \cong_\QQ L_H^*$ then its restriction to $\omega
\times \{0\}$ must have range $\omega \times \{ 0 \}$, and hence it
induces a witness of $L_G \cong_= L_H$: but this implies $G = H$ by
Lemma \ref{lemmaidentity}. Finally, we prove that each $L_G^*$ (where
$G \in \G$) has no nontrivial color preserving automorphism (so that
the map $\Sigma$ assigning to each $G \in \G$ the $Aut(\QQ)$-stabilizer
of $L_G^*$ is constantly equal to the identity group, hence a Borel
map). If $h$ is a color preserving automorphism of $L_G^*$, then it
induces a color preserving automorphism of $L_G$: but since any such
isomorphism must be the identity by Lemma \ref{lemmaidentity}, we must
have $h(n,0) = (n,0)$ for every $n \in \omega$. Moreover, for every $0
\neq q \in \QQ$ and $n \in \omega$ we must have $h(n,q) = (n',q)$ for
some $n' \in \QQ$, since points of the form $(n,q)$ are the only ones
with color $2 \eta(q)$. Finally, since $h(n,0) = (n,0)$ we must have
also $h(n,q) = (n,q)$ for every $q \in \QQ$, because if e.g.\  $h(n,q)
= (n',q)$ with $q > 0$ and $n'$ bigger than $n$ in the order $L_G$,
then there must be some $q'$ strictly  between $0$ and $q$ such that
$h(n,q') = (n',0)$, which is clearly impossible as $q' \neq 0$. Now we
can again apply Theorem \ref{theorsaturation} and get the result.
\end{proof}

A linear order is \emph{non-scattered} if it contains a subset isomorphic to the rationals.
We end this section by extending Theorem \ref{theorQuniversal} to any countable non-scattered linearly ordered set $L$, that is considering the relation $\preceq_L$ of order and color
preserving embedding between colorings on $L$. As in \cite[Corollary 4]{Camerlo}, we define a continuous map from ${}^\QQ
\omega$ into ${}^L \omega$ which assigns to each coloring $c$ on $\QQ$ a
coloring $\hat{c}$ on $L$ and reduces $\preceq_\QQ$ to
$\preceq_L$. First of all, the fact that any non-scattered linear order can be written as a $1+ \QQ +1$ sum of non-empty orders allows to write $L$ as $L' + \sum_{q
  \in \QQ} (U_q + \{ r_q \} + V_q) + L''$, where $U_q$ and $V_q$ are
non-scattered. For $c \in {}^\QQ \omega$ define $\hat{c} \in {}^L
\omega$ by letting $\hat{c}(x) = 0$ if $x \in L' \cup L'' \cup \bigcup
\{ U_q, V_q \mid q \in \QQ\}$, and $\hat{c}(r_q) = c(q) + 1$.

\begin{theorem} \label{theornonscattered}
If $L$ is a countable non-scattered linear order, the quasi-order $\preceq_L$ on ${}^L \omega$ is invariantly universal.
\end{theorem}

\begin{proof}
We will again apply Theorem \ref{theorsaturation} to the map $f \colon
\G \to {}^L \omega \colon G
\mapsto \hat{L}_G^*$, hence it is enough to show that the
hypotheses of that theorem are satisfied. The $\preceq_L$-isomorphism relation
is induced by the natural
action of the group $Aut(L)$ of automorphisms of $L$ on ${}^L
\omega$. To show that $f$ reduces $=_{ \G }$ to $\cong_L$, let $h$ be an isomorphism
between $f(G)$ and $f(H)$. Clearly $h(r_q)$ must be of the form
$r_{q'}$ for some $q' \in \QQ$, as these are the unique points with
non-null color. But then $h$ induces an isomorphism between
$L^*_G$ and $L^*_H$, which in turn implies $G = H$ by the
proof of Theorem \ref{theorQuniversal}.

It remains to show that the map $\Sigma$
assigning to each $G \in \G$ the $Aut(L)$-stabilizer of $f(G) =
\hat{L}^*_G$ is
Borel. Call each suborder of $L$ of the form $L_q = U_q + \{ r_q \} + V_q$ a
\emph{block} (\emph{associated to} $q$), and let $h$ be a color
preserving automorphism
of $f(G)$. We first want to show that for every $q \in \QQ$ there is $q'
\in \QQ$ such that $h \restriction L_q$ is an order and color
preserving isomorphism between $L_q$ and $L_{q'}$ (from this it follows
that $h \restriction L'$ and $h \restriction L''$ are order and color
preserving automorphisms of $L'$ and $L''$, respectively). This follows from
the fact that each block contains exactly one point with color
different from $0$, the image under $h$ of each block (which, in particular, is
an interval) must be an interval, and that each interval of $L$ which
is not included in a single block (and is included neither in $L'$
nor in $L''$) must contain more than one element with non-null color
as the order induced by $L$ on blocks is dense. The second step is to
show that the image of $L_q$ under $h$ must be itself. This is because
if $L_q$ is sent to $L_{q'}$ by $h$, then $h$ induces a color
preserving
automorphism on $L^*_G$ sending $q$ to $q'$: but since any
color preserving automorphism on $L^*_G$ must be the identity by the proof of
Theorem \ref{theorQuniversal}, we get $q=q'$. By the observations
above, $h \colon L\to L$ is a color preserving automorphism of $f(G)$
if and only  if the
following properties hold:
\begin{itemize}
\item $h \restriction L'$ and $h \restriction L''$ are
  automorphisms of $L'$ and $L''$;
\item $h \restriction U_q$ and $h \restriction V_q$ are
  automorphisms of $U_q$ and $V_q$ for every $q \in \QQ$;
\item $h(r_q) = r_q$ for every $q \in \QQ$.
\end{itemize}
Since these conditions are independent of $G$, we have that
all $f(G)$ have the same fixed stabilizer, so that the map $\Sigma$
 is constant (hence,
in particular, Borel).
\end{proof}

\begin{remark}
As observed in \cite{FriMot} about the invariant universality of
embeddability between (ordered) graphs, most of the results of this
paper have an effective counterpart as well, i.e.\ the one obtained by
systematically replacing ``analytic'' and ``Borel'' by (lightface)
$\Sigma^1_1$ and $\Delta^1_1$ in all definitions, statements and
proofs. However, our proof  of Theorem \ref{theornonscattered}
explicitly uses the countably many parameters $Aut(L')$, $Aut(L'')$,
$Aut(U_q)$ and $Aut(V_q)$ to show that the stabilizer map is Borel.
Therefore this proof does not give a proof of the effective counterpart
of Theorem \ref{theornonscattered} (whether such a result holds is
still an open problem), but just of the effective counterpart
\emph{relativized to those parameters}. A similar remark also holds for
Theorem \ref{theoreq2} below.
\end{remark}

\subsection{Dense order preserving embeddings}\label{dop}

Following \cite{MarRos}, a function $g \colon \QQ \to \QQ$ is
\emph{dense order preserving} if it is increasing and for all
$q_0,q_1,r_0,r_1 \in \QQ$ with $f(q_0) < r_0 < r_1 < f(q_1)$ there
exists $q \in \QQ$ such that $r_0 < f(q) < r_1$. This means that ${\rm
  range}(g)$ is dense in the (real) interval $(\inf {\rm range}(g),
\sup {\rm range}(g))$. Moreover, in \cite{MarRos} the dense order
preserving functions were characterized as the restrictions to $\QQ$ of
those continuous increasing functions $g$ from the real line $\RR$
into itself such
that $g(\QQ) \subseteq \QQ$.

Given a quasi-order $R$ on an at most countable set $C$, consider the
analytic quasi-order $\leq_{dop}^R$ on the Polish space ${}^\QQ C$ of
$C$-colorings of $\QQ$ defined by $c \leq_{dop}^R c'$ if and only if
there is a dense order preserving function $g \colon \QQ \to \QQ$ such
that $c(q)\, R\, c'(g(q))$ for every $q \in \QQ$. Extending previous
results of \cite{MarRos}, it was proved in \cite{Camerlo} that when $R$ is one of the
relations $=$ on
$\omega$, $\geq$ on
$\omega$, or $=$ on $2$ (which for simplicity of notation will be
denoted in the sequel by $=$, $\geq$ and $=_2$, respectively) the
corresponding quasi-order $\leq_{dop}^R$
becomes complete analytic. In this section we want to further extend
such results by showing that these quasi-orders are in fact invariantly universal.

The case of $=_2$ implies that of $=$ but, as in
\cite{Camerlo}, dealing with $=$ and $\geq$
simultaneously will provide results for all analytic quasi-orders pinched between $\leq_{dop}^=$ and $\leq_{dop}^{\geq }$. Fix partitions $\{ Q, Q_n \mid n \in \omega \}$ and
$\{ P_n \mid n \in \omega \}$ of $\QQ$ into dense subsets, and for
$\nu \in \seqo \setminus \{ \emptyset \}$, let $Q_\nu = Q_{\nu(0)}
\times \dotsc \times Q_{\nu(|\nu|-1)}$.
If $t \in {}^{l+1} \QQ $, we
call $t$ \emph{good} if there is $\nu \in {}^{l+1} \omega$ such that
$t \in Q_\nu$ (for any good $t$ such a $\nu$ is clearly unique and
will be denoted by $\nu_t$), and \emph{bad} if there is $i \leq l$
such that $t(i) \in Q$. Also, for $n \geq 1$ let $W_n = \{ 1,2,
\dotsc , n\}$.  Finally, let $A =({}^{< \omega} \QQ \setminus \{
\emptyset \} )\times \QQ$ be endowed with
the lexicographic order (where ${}^{< \omega} \QQ$ is ordered
lexicographically and $\QQ$ is ordered by $\leq$, so that $A$ is
isomorphic to $\QQ$ and we can identify them), and for $t \in {}^{< \omega} \QQ \setminus \{
\emptyset \}$ let $J_t = \{ t \} \times \QQ$. We will now construct a
continuous map $G \mapsto G^+$ from the space of graphs on $\omega $ to the space
${}^A \omega$ of colorings on $A$. Let $G$ be given,
and let $(t,q) \in A$ with $|t| = l+1$, some $l \in \omega$.
Let $\alpha(n)$
denote the least $i$ such that $e(i) \in TY_n$.
If $t$ is
good then put:
\begin{itemize}
\item[{\bf -}] $G^+(t,q) = 0$ if $q \in W_{\tau_G(\nu_t) + l}$ and
\item[{\bf -}] $G^+(t,q) = \tau_G(\nu_t) + l + n$ if $q \in P_n \setminus
W_{\tau_G(\nu_t) + l}$.
\end{itemize}
If on the other hand $t$ is bad then set
\begin{itemize}
\item[{\bf -}] $G^+(t,q) = 0$ if $q \in W_{\alpha(l+1) + l}$ and
\item[{\bf -}] $G^+(t,q) =
\alpha(l+2) + l +n$ if $q \in P_n \setminus W_{\alpha(l+1) + l}$.
\end{itemize}

In \cite[Theorem 6]{Camerlo}, it was already proved that if $G$ embeds
into $H$ then $G^+ \leq_{dop}^= H^+$, while if $G^+ \leq_{dop}^\geq
H^+$ then $G$ embeds into $H$. This in particular implies that the map
$G \mapsto G^+$ is a reduction of embeddability between
graphs on $\omega $ to both $\leq_{dop}^=$ and $\leq_{dop}^\geq$.

\begin{theorem}\label{theordopuniversal}
The quasi-orders $\leq_{dop}^=$ and $\leq_{dop}^\geq$ are invariantly
universal. More generally, if $S$ is any analytic quasi-order on
${}^\QQ \omega$ such that $\leq_{dop}^= \subseteq S \subseteq
\leq_{dop}^\geq$ then $(S, \cong_{dop})$ (where $c \cong_{dop} c'$ if
and only if there is a color preserving automorphism of $\QQ$) is
invariantly universal.
\end{theorem}

\begin{proof}
 Identify $\QQ$ with $A$. In order to apply Theorem
 \ref{theorsaturation}, first notice that $\cong_{dop}$ is the
 isomorphism relation associated to both $\leq_{dop}^=$ and
 $\leq_{dop}^\geq$, and that it is induced by the group $Aut(\QQ)$ of
 automorphisms of $\QQ$ (note that any automorphism of $\QQ$ is
 automatically a dense order preserving function since it is
 surjective). Consider now the map $G \mapsto G^+$ previously
 defined. If $h \colon A \to A$ is a color and (dense) order
 preserving isomorphism of $G^+$ to $H^+$, the following hold (see
 the proof of \cite[Theorem 6]{Camerlo} for more details):
\begin{enumerate}[i)]
 \item $\forall t \in {}^{<\omega} \QQ \setminus \{ \emptyset \} \,
   \exists s \in {}^{< \omega} \QQ \setminus \{ \emptyset \} (h(J_t) =
   J_s)$;
\item defining $\rho(t) = s$ if $h(J_t) = J_s$, $\rho$ is an order and
  length preserving bijection;
\item $t$ is good if and only if $\rho(t)$ is good, and in this case
  $\tau_G(\nu_t) = \tau_H(\nu_{\rho(t)})$.
\end{enumerate}
Using these facts as at the end of the proof of \cite[Theorem
6]{Camerlo}, by a back and forth argument we have that $G^+
\cong_{dop} H^+$ implies $G \cong H$. Since for $G,H \in \G$ we have
$G \cong H  \iff G =
H \iff G^+ = H^+$ we get that $G \cong H \iff G^+
\cong_{dop} H^+$.

It remains to prove that the map $\Sigma$ associating to each $G \in \G$ the
$Aut(\QQ)$-stabilizer of $G^+$ is Borel. By an argument similar to
that of Lemma \ref{lemmasimple}, one can check that since $G$ has no
nontrivial automorphisms, then if $h$ is a color preserving
automorphism of $G^+$  and
$\rho$ is defined as in ii) above, $t = \rho(s)$ and $t$ good imply
$\nu_t = \nu_s$, which in turn implies that $(t,q)$ and $(s,q)$ have
same color for every $q \in \QQ$; the same trivially holds in the
case $t$ bad, as by iii) $t$ is good if and only if $\rho(t)$ is
good. Therefore one can check that given a map
$\fhi \colon B \to A$ (where $B$ is a finite subset of $A$), $G^+$
has a color preserving automorphism extending $\fhi$ if and only if the conjunction of
the following Borel conditions is satisfied:
\begin{itemize}
\item $\fhi$ is an order and coloring preserving injection;
\item for every $b \in B$, if $t,s \in {}^{< \omega} \QQ \setminus \{
  \emptyset \}$ are such that $b \in J_t$ and $\fhi(b) \in J_s$ then
\begin{itemize}
\item $|t|=|s|$,
\item $t$ is good if and only if $s$ is good, and if $t$ is good then
  $\nu_t = \nu_s$;
\end{itemize}
\item for every $b,b' \in B$ and $t,s \in {}^{< \omega} \QQ \setminus
  \{ \emptyset \}$, if $b,b' \in J_t$ then $\fhi(b) \in J_s \iff
  \fhi(b') \in J_s$;
\item for every $b\in B$, if $b\in J_t$ with $t$ bad, $ \fhi (b)\in J_s$ and $0<m<|t|$, setting $r=t\restriction m,u=s\restriction m$, one has
\begin{itemize}
\item $r$ is good $ \iff u$ is good,
\item if $r$ is good, then $\nu_r=\nu_u$;
\end{itemize}
\item for every $b,b'\in B$, if $b\in J_t,b'\in J_{t'}, \fhi (b)\in J_s, \fhi (b')\in J_{s'}$ then $t\subseteq t' \iff s\subseteq s'$;
\item for every $t \in {}^{< \omega} \QQ \setminus \{ \emptyset \}$,
  for every $1 \leq i < k_t$ (where $k_t$ is either $\tau_G(\nu_t)
  + l$ or $\alpha(l+1)+l$, depending on whether $t$ is good or bad),
  and for every $(t,q) \in B$, if $\fhi(t,q) = (s,q')$ then
\begin{enumerate}[a)]
\item $q < 1 \iff q' < 1$
\item $i < q < i+1 \iff i < q' < i+1$
\item $q > k_t \iff q' > k_t$
\item $q = i \iff q' = i$ and $q = k_t \iff q' = k_t$. \qedhere
\end{enumerate}
\end{itemize}
\end{proof}

We end this section with the proof that the quasi-order
$\leq_{dop}^{=_2}$ on the space ${}^\QQ 2$ of two-color
colorings of $\QQ$ is already invariantly universal. We first recall
the following construction from
\cite[Theorem 8]{Camerlo}. Let $\{ L_n \mid n \in \omega \}$ be a
collection of linear orders such that no one of them is
isomorphic to an interval of another, and for each $l \in \omega$ fix
scattered subsets
$W_l^{\alpha(l+1)}, W_l^{\alpha(l+1)+1}, \dotsc, W_l^{\alpha(l+2)-1}$
of $\QQ$ pairwise incomparable under embeddability.
For $t\in {}^{<\omega } \QQ \setminus\{\emptyset\} $, set
$J_t=L_{|t|-1}\times 2\times \QQ $ and let
$A= \{ (t,u,i,r) \mid t \in {}^{< \omega} \QQ \setminus \{ \emptyset
\}, (u,i,r) \in J_t\}$ be endowed with the lexicographic order, so
that $A$ is isomorphic to
$\QQ$. Define $t$ to be good or bad as done before. Given a countable
graph $G$, define $G'$ to be the coloring on $A$ defined by letting
$(t,u,i,r)$ to have color $i$ if and only if:
\begin{itemize}
\item[-] either $t$ is bad,
\item[-] or else $t$ is good and either $i = 0$ or $r \notin
  W_{|t|-1}^{\tau_G(\nu_t)}$.
\end{itemize}
Otherwise (that is if $t$ is good, $i=1$ and $r\in W_{|t|-1}^{\tau_G(\nu_t)}$) let
the color be $0$.
In \cite[Theorem 8]{Camerlo}  it is proved that the continuous map $G
\mapsto G'$ reduces embeddability between countable graphs to
$\leq^{=_2}_{dop}$.

\begin{theorem}\label{theoreq2}
The quasi-order $\leq_{dop}^{=_2}$ is invariantly
universal. Therefore, if $R$ is a quasi-order on $\omega$ with at
least two incomparable elements, then $(\leq^R_{dop},\cong_{dop})$ is
invariantly universal.
\end{theorem}

\begin{proof}
 The application of Theorem \ref{theorsaturation} will suffice again,
 so we have just to show that the hypotheses of that
 theorem are satisfied. Identify $A$ with $ \QQ $. Clearly the
 isomorphism relation associated to $\leq_{dop}^{=_2}$ is
  $\cong_{dop}$ again, so it is induced by the natural action of
  $Aut(\QQ)$. That the map $G \mapsto G'$ reduces $\cong_{ \G }$ to $\cong_{dop}$ essentially follows from
 the fact that properties similar to conditions
 i)--iii) given in the proof of Theorem \ref{theordopuniversal} hold
 also with respect to our new construction, so we
 need just to prove that the map $\Sigma$ sending each $G \in \G$ to the
 $Aut(\QQ)$-stabilizer of $G'$ is Borel. Arguing as before,
 since $G$ has no nontrivial automorphisms then if $h$ is a color preserving
 automorphism of $G'$ and $\rho$ is defined as in ii)
 above, $s= \rho(t)$ and $t$ good imply $\nu_t = \nu_s$, which in
 turn implies that $(t,u,i,r)$ and $(s,u,i,r)$ have
 same color for every $(u,i,r) \in L_{|t|-1} \times 2 \times \QQ$
 (the same is trivially true for bad $t$ and $s$
 again). Moreover, it is not hard to check that for every $(t,u,i,r),
 (t,u',i',r') \in A$ if $h(t,u,i,r) = (s,v,j,p)$ and
 $h(t,u',i',r') = (s,v',j',r')$ then $i=j$ (which in particular
 implies $i= i' \imp j = j'$) and $u = u' \imp v = v'$.
 Therefore we have that given $\fhi \colon B\to A$, where $B$ is a
 finite subset of $A$, $G'$ has a color preserving automorphism
 extending $\fhi$ if
 and only if the conjunction of the following Borel conditions is
 satisfied:
\begin{itemize}
 \item $\fhi$ is an order and coloring preserving injection;
 \item if $(t,u,i,r),(t',u',i',r') \in B$, $\fhi(t,u,i,r) = (s,v,j,p)$
   and $\fhi(t',u',i',r') = (s',u',i',p')$, then
\begin{itemize}
\item $|t|=|s|$,
\item $t$ is
good if and only if $s$ is,
\item if $t$ is good then $\nu_t = \nu_s$,
\item $i=j$,
\item $t = t' \imp s=s'$,
\item $t = t' \wedge u = u' \imp s=s'
\wedge v = v'$;
\end{itemize}
\item for every $(t,u,i,r)\in B$ with $t$ bad, if $ \fhi (t,u,i,r)=(s,v,i,p),0<m<|t|,t'=t\restriction m,s'=s\restriction m$, then
\begin{itemize}
\item $t'$ is good $ \iff s' $ is good,
\item if $t'$ is good then $\nu_{t'}=\nu_{s'}$;
\end{itemize}
\item if $(t,u,i,r),(t',u',i',r')\in B, \fhi (t,u,i,r)=(s,v,i,p), \fhi (t',u',i',r')=(s',v',i',p')$, then $t\subseteq t' \iff s\subseteq s'$;
\item for each $t \in {}^{l+1} \QQ$, $\fhi_t \colon {\rm dom} \fhi_t \subseteq L_l \to L_l$ can
  be extended to an automorphism of $L_l$, where
$\fhi_t$ is defined on $u$ with value $v$ if there is $(t,u,i,r) \in
B$ such that $\fhi(t,u,i,r)$ is of the form
$(s,v,i,p)$ for some $s,p$ ($\fhi_t$ is well-defined when the
previous conditions are satisfied);
\item for each  $t \in {}^{l+1} \QQ, u \in L_l$, the map
  $\fhi_{t,u} \colon \QQ \to \QQ$ can be extended to an
automorphism of $\QQ$ which setwise fixes $W_t$, where
$\fhi_{t,u}(r) = p$ if there is $(t,u,1,r) \in B$ such that
$\fhi(t,u,1,r) = (s,v,1,p)$ for some $s,v$, $W_t= \emptyset$ if $t$ is bad, and
$W_t= W^{\tau_G(\nu_t)}_l$ if $t$ is good. \qedhere
\end{itemize}

\end{proof}

\subsection{Continuous embeddings between compacta}\label{dendrites}

In this section, Theorem \ref{theorsaturation} will be applied to show
the invariant universality of the pair $(\sqsubseteq_c,\simeq )$,
where $\sqsubseteq_c$ is the relation of continuous embeddability
between dendrites (that is, compact, connected, locally connected
metric spaces not containing circles) and $\simeq $ is homeomorphism.
Dendrites form a standard Borel space, in fact a $ \boldsymbol
\Pi^0_3$-complete subset of the hyperspace $K({}^{\omega }[0,1])$, by
results in \cite{CDM2005}.
If a dendrite is the union of finitely many arcs, then it is called a tree.
The order $ord(p,X)$ of a point $p$ in a topological space $X$ is the
smallest cardinal $\kappa $ such that $p$ has an open neighborhood basis in
$X$ whose members have boundaries of cardinality at
most $\kappa $.

For a point $p$ in a dendrite $D$, the inequality $ord(p,D)\leq\omega $ holds.
If $ord(p,D)=1$ then $p$ is called an end point of $D$; if
$ord(p,D)\geq 3$, then it is called a branch point.
A maximal open free arc in $D$ is a subset of $D$ homeomorphic to
$(0,1)$ that does not contain branch points of $D$ and is maximal with
these properties.
What we shall actually show is the invariant universality restricted
to the class $Z$ of dendrites having infinitely many maximal open free
arcs, infinitely many branch points and whose points have order at
most $3$ (the fact that this is a Borel set can be recovered from the
proofs of \cite[Lemma 6.5]{CDM2005} and of \cite[Lemma 1.4]{MarRos}).
Again by \cite{CDM2005}, the set of pairs $(D,p)\in Z\times {}^{\omega
}[0,1]$ such that $p$ is a branch point of $D$ is a Borel set (with
countable vertical sections) and the same for the set of $(D,A)\in Z\times
K({}^{\omega }[0,1])$ such that $A$ is the closure of a maximal open
free arc in $D$.

As discussed after Theorem \ref{theorsaturation}, we need now: a map
$f$ from countable graphs to $Z$ and a map $g$ reducing $\simeq $ to a
suitable orbit equivalence relation.
Function $f$ has already been defined in \cite{Camerlo}, and we shall
recall here its definition; its range will actually
yield dendrites contained in the square ${}^2[0,1]$.
As for $g$, we shall
partially modify a construction of \cite{CDM2005}, since this will
make it easier to check the properties we need.

First we define the function $f$ and prove that it has the desired
properties. Fix trees $T_0,T_1,T_2\in Z$ with the property that
\begin{itemize}
\item[{\bf -}] the order of points in each $T_j$ is at most $3$;
\item[{\bf -}] each
$T_j$ has an end point $p_j$ such that for $i\neq j$ there is
no continuous embedding $ \fhi \colon T_i\to T_j$ with $ \fhi (p_i)=p_j$.
\end{itemize}
Fix also an enumeration $\langle  q_n \mid n\in\omega \rangle$ of $
\QQ \cap ( \frac 12 ,1)$ and an increasing bijection $h\colon  \QQ \cap (
\frac 12 ,1)\to \QQ
$.
For $(n,j)\in\omega\times 2$, let $g_{nj}\colon T_j\to [ \frac 12 ,1]\times
[0,1]$ be a continuous embedding whose range $T_n^j$ has diameter less
that $ \frac 1{n+1} $ and such that $g_{nj}(p_j)=(q_n,0)$ and
$T_n^j\cap T_{n'}^{j'}=\emptyset $ for $n\neq n'$.
Similarly, fix a continuous embedding $ \bar g \colon T_2\to [0, \frac 12
)\times [0,1]$ with $ \bar g (p_2)=( \frac 14 ,0)$ ant let $T$ be
its range.

Now, given a graph $G$ on $\omega $, let $G'$ be as in Theorem
\ref{theoreq2} and let
$$f(G)=([0,1]\times\{ 0\} )\cup T\cup\bigcup_{n\in\omega }T_n^{\gamma_n},$$
where $\gamma_n\in 2$ is the color assigned by $G'$ to $h(q_n)$. Notice
that $f(G)\in Z$.

The fact that $f$ reduces embeddability of graphs on $\omega $ to
$\sqsubseteq_c$ is contained in \cite{Camerlo}.
Suppose that $f(G)\simeq f(H)$, for $G,H\in \G $; then
$G'\cong_{dop}H'$ and $G=H$ by the previous section.

Now we define the function $g$. Let $\Lambda =\{ P,R\} $, where $P$ is
a unary relation symbol and $R$ is a ternary relation symbol.
In \cite{CDM2005} a function $\Phi $ is built reducing homeomorphism
on $Z$ (actually on a bigger class) to isomorphism in $Mod_{\Lambda
}$, the Polish space of $\Lambda $-structures on $\omega$: roughly
speaking, given $D \in Z$ the natural numbers (which constitute the
domain of $\Phi(D)$) are used as codes for the branching points and the
closure
of the maximal open free arcs of $D$, the interpretation in $\Phi(D)$
of the predicate $P$ identify the (codes for) branching points of $D$,
and the ternary relation $R$ is interpreted in $\Phi(D)$ as the
relation of ``being in between''.
We use here a similar reduction $g \colon Z\to Mod_{\Lambda }$ by taking
$g\restriction (Z\setminus f( \G ))=\Phi\restriction (Z\setminus f( \G
))$ and defining $g$ on $f( \G )$ in such a way to keep $g(D)$
isomorphic with $\Phi (D)$, so to grant that $g$ will still be a
reduction.

Let $\beta_n \colon f( \G )\to {}^2[0,1]$ be a sequence of Borel functions
such that $\beta_n(D)$ enumerate the branch points of $D$ having
positive second coordinate, so that the set of branch points of $D$ is
$\{ ( \frac 14,0)\}\cup (( \QQ \cap ( \frac 12 ,1))\times\{ 0\})\cup\{\beta_n(D)\mid n\in\omega\} $.
Similarly, let $\alpha_n \colon f( \G )\to K({}^2[0,1])$ be a sequence of
Borel functions such that $\alpha_n(D)$ enumerate the closures of the
maximal open free arcs of $D$ not lying on $[0,1]\times\{ 0\} $.
The collection of closures of maximal open free arcs of $D$ is thus
$\{ [0, \frac 14 ]\times\{ 0\} ,[ \frac 14 , \frac 12 ]\times\{ 0\}\}\cup\{\alpha_n(D)\mid n\in\omega\} $.
Now let $p_n \colon f( \G )\to K ({}^2[0,1])$ be defined by letting
$$p_n(D)= \left \{
\begin{array}{lcl}
\{ ( \frac 14 ,0)\} & \mbox{if} & n=0, \\
\{\beta_m(D)\} & \mbox{if} & n=4m+2, \\
\{ (q_m,0)\} & \mbox{if} & n=4m+4, \\
{}[0, \frac 14 ]\times\{ 0\} & \mbox{if} & n=1, \\
{}[ \frac 14 , \frac 12 ]\times\{ 0\} & \mbox{if} & n=3, \\
\alpha_m(D) & \mbox{if} & n=2m+5.
\end{array}
\right .
$$
So $\{ p_n(D)\mid n\in\omega\} $ is an enumeration of branch points and closures of maximal open free arcs of $D$.

To define function $g$ on $f( \G )$ consider $D\in f( \G )$.
For $n,m,m',m''\in\omega $, let
\begin{itemize}
\item $P^{g(D)}(n)$ if and only if $n$ is even,
\item $R^{g(D)}(m,m',m'')$ if and only if all points in $p_{m'}(D)$
  lie on any path from a point in $p_m(D)$ to a point in
  $p_{m''}(D)$.
\end{itemize}
As in the proof of \cite[Lemma 6.5]{CDM2005}, function $g$ is Borel
(since $f(\G)$ is Borel by the injectivity of $f$) and reduces
homeomorphism on $Z$ to isomorphism on
$Mod_{\Lambda }$.

\begin{theorem}
The pair $(\sqsubseteq_c,\simeq )$ on dendrite is invariantly
universal: for any analytic quasi-order $R$ there is a Borel class $
\mathcal C $ of dendrites closed under homeomorphism such that $R$ is
Borel bireducible with the restriction of $\sqsubseteq_c$ to $ \mathcal C
$.
\end{theorem}

\begin{proof}
For $l\neq 1,3$ let $\kappa_l \colon f( \G )\to\{ 4m\mid m\in\omega\} $ be defined by setting $\kappa_l(D)=n$ if and only if $p_n(D)$ is the value that the first point map on $[0,1]\times\{ 0\} $ assumes on some (all) points of $p_l(D)$.
Notice the following facts:
\begin{enumerate}
\item Functions $\kappa_l$ are Borel:
  indeed, $\kappa_l(D)=n\Leftrightarrow\neg\exists r\in\{ 4m\mid
  m\in\omega\}\ (r\neq n\wedge R^{g(D)}(l,r,n))$.
\item Denoting for $D\in f( \G ),n\in\omega $ by $L_D(n)\in 2$
  the color assigned by $(f^{-1}(D))'$ to $h(q_n)$, the map $f( \G
  )\to {}^{ \omega }2 \colon D\mapsto L_D$ is Borel.
\end{enumerate}

By the observations preceding this theorem, in order to apply Theorem
\ref{theorsaturation} it remains only to show that the map $\Sigma' \colon f( \G )\to G(S_{\infty
})$ assigning to each $D\in f( \G )$ the
stabilizer $Stab(g(D))$ of $g(D)$ is Borel (so that the map $\Sigma$
assigning $Stab(g(f(G)))$ to each $G \in \G$, which is simply the
composition of $f$ and $\Sigma'$, is Borel as well).
So fix a finite injective sequence $s=(s_0,\ldots ,s_k)$ with $k\geq
3$ of natural numbers, in order to check that
 \[ V_s=\{ D\in f( \G )\mid\exists \fhi \in Stab(g(D))\ (s\subseteq \fhi) \} \]
is Borel.
Indeed, $D\in V_s$ is equivalent to the following
Borel conditions on $D$:
\begin{itemize}
\item $s_0=0,s_1=1,s_2=2$;
\item $P^{g(D)}(i)\Leftrightarrow P^{g(D)}(s(i))$, for $i\in\{ 0,\ldots ,k\} $;
\item $R^{g(D)}(i,i',i'')\Leftrightarrow R^{g(D)}(s_i,s_{i'},s_{i''})$, for $i,i',i''\in\{ 0,\ldots ,k\} $;
\item
  $\kappa_i(D)=\kappa_{i'}(D)\Leftrightarrow\kappa_{s_i}(D)=\kappa_{s_{i'}}(D)$, for $i,i' \in\{ 0,\ldots ,k\}\setminus\{ 1,3\} $;
\item for any $n$ such that $\kappa_{\bar{l}}(D) = n$ for some
  $\bar{l} \leq k$, the restriction of $s$ to $\{ l\leq
  k\mid\kappa_l(D)=n\} $ is extendible to an isomorphism $\psi $
  between the finite structures $ \mathcal A =\{
  l\in\omega\mid\kappa_l(D)=n\} $ and $ \mathcal B =\{
  l\in\omega\mid\kappa_l(D)=\kappa_{s_{ \bar l }}(D)\} $;
\item there is an automorphism of $g(D)$ extending the function
  sending $\kappa_l(D)$ to $\kappa_{s_l}(D)$ for all $l\in\{ 0,\ldots ,k\}\setminus\{ 1,3\} $.
\end{itemize}
This last set is Borel, since it is the set of all $D\in f( \G )$ such
that there is a color preserving automorphism of
$(f^{-1}(D))'$ extending the function that sends each $h(q_{
  \frac{\kappa_l(D)-4}4 })$ to
$h(q_{ \frac{\kappa_{s_l}(D)-4}4 })$ when $\kappa_l(D)\neq 0$; thus this follows from the proof
of Theorem \ref{theoreq2}.
\end{proof}

\subsection{Isometric embeddings between Polish metric spaces}\label{metric}

Now we consider the case of isometric embeddability on (some classes
of) Polish metric spaces. Recall that $(X,d)$ is a \emph{Polish metric space}
if $d$ is a complete metric on $X$ which generates a second countable
topology on $X$. It is a well-known fact that there is a universal
Polish space  $\mathbb{U}$ (called \emph{Urysohn space}) such that each
Polish metric space is isometric to a closed subspace of $\mathbb{U}$
and $\mathbb{U}$ is ultrahomogeneous
(such a $\mathbb{U}$ is unique up to isometry). Therefore it is
natural to consider the space $F(\mathbb{U})$ of the closed subspaces
of $\mathbb{U}$ (endowed with the usual Effros Borel structure, see \cite[\S 12.C]{Kechris1995}) as the
standard Borel space of Polish metric spaces. As shown in \cite{louros}, the relation of isometric embeddability $\sqsubseteq^i$ (that is the
quasi-order on $F(\mathbb{U})$ induced by isometric embeddings, i.e.\ distance preserving functions, between
Polish metric spaces)
is a complete analytic quasi-order (and $\cong^i$, the relation of
being isometric, is clearly the $\sqsubseteq^i$-isomorphism). To be
more precise, Louveau and Rosendal
gave two different proofs of the completeness of $\sqsubseteq^i$ by
considering two specific subclasses of $F(\mathbb{U})$: the class of
\emph{discrete} Polish metric spaces (i.e.\ spaces in which the induced
topology is discrete) and \emph{ultrametric} Polish spaces (i.e.\
spaces in which the metric $d$ is actually an ultrametric, that is $d$
satisfies the inequality $d(x,z) \leq \max \{ d(x,y),d(y,z) \}$ for all
points $x,y,z$ of the space).  It is natural to ask whether $\sqsubseteq^i$ is
invariantly universal, and even whether its restrictions to the subclasses
above are invariantly universal (such restricted versions easily imply
the general result, as the notions of discrete Polish metric space and
ultrametric Polish space are invariant under isometry).

We first consider the easier case of discrete Polish metric
spaces. Note that this class was also considered in
\cite{FriMot}, although with a different coding: in that paper the
class of discrete Polish metric spaces was defined as the collection
of the spaces of the form $(\omega,d)$ where $d$ induces a discrete
topology on $\omega$ (this can be done since any discrete Polish metric space must
be countable), while here we want to consider the class $\mathcal{D}
\subseteq F(\mathbb{U})$ of those infinite spaces $(X,d)$ such that $d$ induces a
discrete topology, which is potentially a much harder problem.
Let
  us note here that $\mathcal{D}$ is \emph{not} a standard Borel space
  as it is a proper co-analytic subset of $F(\mathbb{U})$ --- see
  \cite[Exercise 27.8]{Kechris1995}. In fact, we shall prove the result for a Borel $Z\subseteq \mathcal D $ invariant under isometry.

Given a combinatorial tree $G$ on $\omega$, consider the discrete Polish metric space $D_G$
whose domain is $\omega$ and whose distance $d_G$ is
the geodesic distance on $G$, that is $d_G(n,m) = $ the length of the (unique) path joining
$n$ to $m$ in $G$. Such a $D_G$ will be then
Borel-in-$G$ coded as an element of $\mathcal{D} \subseteq F( \mathbb U )$, also denoted by
$D_G$. As noted in \cite{louros}
and \cite{FriMot}, from such a $D_G$ one can recover the structure of
$G$ by linking two vertices $n,m$ if and only if $d_G(n,m) = 1$, so
that we clearly have $G \cong H \iff D_G \cong^i D_H$ and $G
\sqsubseteq H \iff D_G \sqsubseteq^i D_H$ for each pair of
combinatorial
trees $G,H$.

In order to apply Theorem \ref{theorsaturation}, we now define
a reduction of $\cong^i$ on $\mathcal{D}$ to isomorphism on a suitable
class of structures (the existence of such a reduction is a well-known
result, see e.g.\ \cite{gaokechris}, but here we need an explicit
definition). Recall from \cite[Theorem 12.13]{Kechris1995} that
there is a sequence $\psi_n \colon F(\mathbb{U}) \to \mathbb{U}$ of
Borel functions such that for every $F \in F(\mathbb{U})\setminus\{\emptyset\} $ the sequence
$\langle \psi_n(F) \mid n \in \omega \rangle$ is an enumeration
(which can be assumed without repetitions if $F$ is infinite) of a dense subset of $F$. If we consider the
restriction of those $\psi_n$ to $\mathcal{D}$, the sequence defined
above will actually be an enumeration of the full $F \in \mathcal{D}$, as each
point of $F$ is isolated.

Consider the infinite
language $\Lambda = \{ R_q\mid q\in \QQ^+\}$, where each $R_q$ is a
binary predicate.
To each $D = (X,d) \in
\mathcal{D}$ associate the $\Lambda $-structure $S(D)$ on $\omega$ putting
$R_q(i,j) \iff d(\psi_i(X),\psi_j(X)) < q$. It is clear that if $D =
(X,d),D' = (X',d') \in \mathcal{D}$ and $\fhi$ is an isometry between
$D$ and $D'$ then the map $f_\fhi \colon \omega \to \omega$ which maps
$i$ to the unique $j$ such that $\fhi(\psi_i(X)) = \psi_j(X')$ is an
isomorphism between $S(D)$ and $S(D')$. Conversely, if $\rho$ is an
isomorphism between $S(D)$ and $S(D')$ then the map $g_\rho \colon X
\to X' \colon \psi_n(X) \mapsto \psi_{\rho(n)}(X')$ is an isometry
between $D$ and $D'$.

\begin{theorem}\label{theordiscrete}
The relation of isometric embeddability between discrete Polish metric
spaces is invariantly universal, meaning that for every analytic
quasi-order $R$ there is a Borel class $\mathcal{C} \subseteq \mathcal{D}$
closed under isometry such that $R \sim_B {{\sqsubseteq^i}
  \restriction \mathcal{C}}$.

In particular, $\sqsubseteq^i$ between arbitrary Polish metric spaces
is invariantly universal.
\end{theorem}

\begin{proof}
Let $Z=\{ F\in F( \mathbb U )\mid F \mbox{ is infinite} \wedge\forall n,m\in\omega\ d(\psi_n(F),\psi_m(F))\in\omega\} $.
So $Z$ is Borel and closed under isometry.
Moreover, $Z\subseteq \mathcal D $:
indeed, for any $F\in Z$, notice that $F=\{\psi_n(F)\mid n\in\omega\} $ since $\{\psi_n(F)\mid n\in\omega\} $ is closed
in $ \mathbb U $, its only Cauchy sequences being the eventually constant ones; this same reason entails that $F$ is discrete.

It is
enough now to show that the hypotheses of Theorem \ref{theorsaturation}
are satisfied. As
already observed above, the map which sends a combinatorial tree $G$ to
the discrete Polish metric space $D_G\in Z$ simultaneously reduces
$\sqsubseteq$ to $\sqsubseteq^i$ and $\cong$ to $\cong^i$. Moreover
the map which sends $D\in Z$ to the structure
$S(D)$ is a reduction of $\cong^i$ to $\cong$, and since $\cong$ is
induced by the
logic action of $S_\infty$ it remains to show that the map $\Sigma
\colon \G
\to G(S_\infty)$ which associates to each $G \in \G$ the group
$H_G$ of the
automorphisms of $S(D_G)$ is Borel. But it is easy to check that any
nontrivial automorphism of $S(D_G)$ induces a nontrivial isometry of
$D_G$ into itself, which in turn induces a nontrivial automorphism of
$G$: thus, since by Corollary \ref{cornontrivial} any $G \in \G$ is
rigid, we get
that each $H_G$ consists exactly of the
identity function, therefore the  map $G \mapsto H_G$ is constant
(hence Borel).
\end{proof}

We now consider the slightly more complicated case of
ultrametric Polish spaces. In this case the collection $\mathcal{U}$
of all infinite ultrametric Polish spaces forms a
standard Borel subspace of $F(\mathbb{U})$. We shall again apply
Theorem \ref{theorsaturation}, so we need a Borel reduction $f \colon \G \to
\mathcal{U}$ of $\sqsubseteq$ into $\sqsubseteq^i$ and, simultaneously, of
$\cong$ into $\cong^i$ and a Borel function $g$ reducing $\cong^i$ on $\mathcal{U}$
to some orbit equivalence relation in a suitable way.

The function $f$ was (essentially) already defined in \cite{louros} and \cite{FriMot}:
for $G \in \G$, let $U_G$ be the ultrametric Polish space consisting
of all maximal paths of $G$ beginning in the root $\emptyset$ (such
paths will be called \emph{branches} of $G$),
together with the metric $d_G$ defined by $d_G(x,y) = 2^{-n}$, if
branches $x$ and $y$ are distinct and share $n+1$ vertices of $G$, and
$d_G(x,y) = 0$ if $x=y$. Each $U_G$ is easily seen to be an
ultrametric Polish
space, and is isometrically identified in a Borel
way with an element of $\mathcal{U}$, also denoted by $U_G$. Given any $G \in \G$, denote by $a_x$, $x \in {}^\omega \omega$, the
 unique maximal branch which contains all vertices of the form $x
 \restriction n$, by $b_{s, i}$ the branch determined by the points
 $(s^{++},i,0)$ (whenever such points exist in $G$), and by
 $c_{s,u,i}$, $u \in {}^{< \omega} 2, i = 0,1$, the branch determined
 by the points $(u,s,0^{2 \theta(u) +2} {}^\smallfrown i)$ (whenever
 these points exist in $G$). Let us call a \emph{fork} any set of
 branches of the form $F^G_s = \{ b_{s, i} \mid i \leq \# s + 2 \}$ or
  $F^G_{s,u} = \{ c_{s,u,i} \mid i = 0,1 \}$ (where $s \in
 {}^{< \omega} \omega, u \in {}^{< \omega}2$). Note that a point of
 $U_G$ is non-isolated if and only if is of the form $a_x$.

\begin{lemma} \label{lemmaisometry1}
Let $G,H \in \G$ and $U_G, U_H$ be defined as above. Then
 $U_G \cong^i U_H$ implies $G = H$.
\end{lemma}

\begin{proof}
Note that  any isometry  $h$ between $U_G$ and $U_H$ must be a bijection of  $\{ a_x \mid x \in {}^\omega \omega \}$ onto itself. It is sufficient to show that

\begin{claim}
Given $s \in {}^{<\omega} \omega, u \in {}^{<\omega} 2$, $F^G_{s,u}$ is a fork of $U_G$ if and only if $F^H_{s,u}$
is a fork of $U_H$.
\end{claim}

To prove the claim we first show that $h(F^G_s) = F^H_s$. Note that all elements of $F^G_s$
are isolated, have distance $2^{-2 |s|}$ from $a_{s {}^\smallfrown \vec{0}}$, have distances
$2^{-(2|s|+2)}$ between them, and are the only points of $U_G$ with
these properties. Let $t = y \restriction |s|$ where $a_y = h(a_{s {}^\smallfrown \vec{0}})$: since $h$ is an isometry and the elements
of $F^H_t$ are exactly the isolated points in $U_H$ with distance
$2^{-2 |s|}$ from $a_y$ and distance $2^{-(2|s|+2)}$ between them, we
conclude that $h(F^G_s) = F^H_t$. In particular, $F^G_s$ and $F^H_t$ have the same cardinality, hence $s=t$.

This also implies that $h(a_x) = a_x$ for every $x \in {}^\omega \omega$.
Now note that the elements of
$F^G_{s,u}$ are exactly the two points of $U_G$ which have
distance $2^{-2 |s|}$ from $a_{s {}^\smallfrown \vec{0}}$, and have distance
$2^{-(2|s|+2\theta(u)+3)}$ between them. But then $h$ maps them to two points of $U_H$ which have distance $2^{-2 |s|}$ from
$a_{s {}^\smallfrown \vec{0}}$, and have distance $2^{-(2|s|+2\theta(u)+3)}$ between them. Such points can only be the elements of $F^H_{s,u}$, which is therefore a fork of $U_H$.
\end{proof}

\begin{lemma} \label{lemmaisometry2}
Let $G,H \in \G$. Then
 $G \sqsubseteq H \iff U_G \sqsubseteq^i U_H$.
\end{lemma}

\begin{proof}
We just show $U_G \sqsubseteq^i U_H \imp G \sqsubseteq H$
because any embedding between  $G$ and $H$ can be easily converted
into an isometric embedding between $U_G$ and $U_H$.

Fix an isometric embedding $h$ from $U_G$ to $U_H$. Note that
 $h$ must send
 elements of the form $a_x$ into elements of
 the same form. The function $h$ induces the Lipschitz map $\fhi \colon {}^{< \omega} \omega \to {}^{< \omega} \omega$ defined by $\fhi(s) = y \restriction |s|$ where $a_y = h(a_{s {}^\smallfrown \vec{0}})$.

We first show that if $s \in {}^{<\omega} \omega, u \in {}^{<\omega} 2$ are
such that $F^G_{s,u}$ is a fork of $U_G$, then $F^H_{\fhi(s),u}$
is a fork of $U_H$.
In fact, as already noticed in the proof of the previous lemma, the elements of
$F^G_{s,u}$ are exactly the two points of $U_G$ which  have
distance $2^{-2 |s|}$ from $a_{s{}^\smallfrown \vec{0}}$, and have distance
$2^{-(2|s|+2\theta(u)+3)}$ between them. But then there must be two
points in $U_H$ which have distance $2^{-2 |s|}$ from
$a_y$ (where $y$ is as in the definition of $\fhi$), and have distance $2^{-(2|s|+2\theta(u)+3)}$ between them: this can
happen only if $F^H_{\fhi(s),u}$ is a fork of $U_H$.

Now let $S$ and $T$ be normal trees such that $G = G_S$ and $H = G_T$. By the above, $\fhi$ is a witness of $S \leq_{max} T$, and hence $G \sqsubseteq H$.
\end{proof}

Regarding the map $g$ that will be used in Theorem
\ref{theorultrametric}, it was already introduced in
\cite{gaokechris}, although here we need to perform a slight
modification to the original construction in order to simplify
subsequent computations. Let $\psi_n \colon F(\mathbb{U}) \to
\mathbb{U}$ be Borel functions defined as above.
The basic open balls of a space $U \in \mathcal{U}$ are
of the form $B(\psi_n(U),q) = \{ x \in U \mid d_U(x, \psi_n(U)) < q
\}$ for some $n \in \omega$ and $q \in
\QQ^+$. Let $\Lambda $ be a first order language consisting of a binary
relation symbol $R$ and of countably many unary relation symbols
$Q_r$, for $r\in \QQ^+$. For
each $U \in \mathcal{U}$, consider the $\Lambda $-structure $S(U)$ whose
domain is $\{ (n, q) \mid n\in \omega ,q\in \QQ^+ \}$, in which $R$
holds between $(n,q)$ and $(n',q')$ if and
only if $B(\psi_n(U),q) \subseteq B(\psi_{n'}(U),q')$, and in
which $Q_r$ holds for $(n,q)$ if and only if ${\rm
  diam}(B(\psi_n(U),q)) <r$. As it is shown in \cite{gaokechris},
the map sending $U \in \mathcal{U}$ to $S(U)$ (coded as a
$\Lambda $-structure with domain $\omega$) is Borel and reduces
$\cong^i$ to $\cong$.

Notice now that the above construction does not really depend on the
choice of the functions $\psi_n$: had we chosen a different sequence
of functions $\psi'_n$ (with the same properties), the structure
constructed from $U \in \mathcal{U}$ as explained above but using the
$\psi'_n$ instead of the $\psi_n$ would be formally different but
still isomorphic to $S(U)$. So define the $\psi'_n$ as follows: first set $\psi'_n \restriction (F(\mathbb{U}) \setminus f(\G))= \psi_n
\restriction (F(\mathbb{U}) \setminus f(\G))$.
Then, for $G \in \G$ let
\begin{itemize}
\item[{\bf -}] $\psi'_{3 \# s}(U_G) = a_{s {}^\smallfrown \vec{0}}$,
\item[{\bf -}] $\langle \psi'_{3n +
  1}(U_G) \mid n \in \omega \rangle$ be the enumeration of the branches of $G$ of the form $b_{s ,i}$ according to the position of $(s,i)$ in some fixed ordering of ${ }^{<\omega }\omega\times\omega $ in type $\omega $,
\item[{\bf -}] $\langle \psi'_{3n+2}(U_G)\mid n \in
\omega \rangle$ be the enumeration of the branches of $G$ of the form
$c_{s,u,i}$ according to the position of $(s,u,i)$ in some fixed ordering of ${ }^{<\omega }\omega\times{ }^{<\omega }2\times 2$ in type $\omega $.
\end{itemize}
Since $f(\G)$ is a Borel subset
of $\mathcal{U}$ (as $\G$ is Borel and $f$ is injective because it reduces
equality on $\G$ to isometry), the
functions $\psi'_n$ are still Borel and their values on $U \in \mathcal{U}$
still enumerate a dense subset of $U$. Still call $S(U)$ the $\Lambda $-structure
constructed from $U \in \mathcal{U}$ as explained above but using these
specific $\psi'_n$ in the construction (each $S(U)$ will be confused
with its Borel-in-$U$ coded version as a $\Lambda $-structure on
$\omega$), and let $g$ be the Borel map sending $U$ to $S(U)$.

\begin{theorem}\label{theorultrametric}
The relation of isometric embeddability between ultrametric Polish
spaces is invariantly universal, that is for every analytic
quasi-order $R$ there is a Borel class $\mathcal{C} \subseteq \mathcal{U}$
closed under isometry such that $R \sim_B {{\sqsubseteq^i}
  \restriction \mathcal{C}}$.
\end{theorem}

\begin{proof}
To show that the hypotheses of Theorem \ref{theorsaturation} are
satisfied, consider the maps $f$ and $g$ defined above: by Lemmas \ref{lemmaisometry1} and \ref{lemmaisometry2} and the previous
observations, we just need to check that the map $\Sigma$ sending $G \in \G$
into the group of automorphisms of $S(U_G)$ is
Borel.

First note that, by the proof of Lemma \ref{lemmaisometry1},  an isometry of $U_G$ into itself must
be the identity on branches of the form $a_x$ (i.e.\ on its accumulation points) and set-wise fix the
forks of $U_G$. Then notice that the
function taking $G$ into the binary relation $D_G$ on $\omega$ defined
by
$$\begin{array}{rcl}
n \, D_G\,  m & \iff & \psi'_n(U_G) \mbox{ and } \psi'_m(U_G) \mbox{ are both
isolated} \\
 & & \mbox{and belong to the same fork of } G
\end{array}
$$
is Borel. Finally, recall
from \cite{gaokechris} that any automorphism $H$ of $S(U_G)$ induces
the isometry $H' \colon U_G \to U_G$ defined by $H'(x) = y \iff \{y\}
= \bigcap_{k \in \omega} B(\psi_{m_k}(U_G),q_k)$, where the $m_k$ and the $q_k$ are such that there are sequences $\langle n_k \mid k
\in \omega \rangle, \langle r_k \mid k \in \omega \rangle$ of natural
numbers and elements of $\QQ^+$, respectively, with the properties
that $\psi_{n_k}(U_G) \to x$, $r_k \to 0$, and $H(n_k,r_k) =
(m_k,q_k)$ for every $k \in \omega$. From all these facts it follows
that given a function $h \colon A \subseteq \omega \times \QQ^+  \to
\omega \times \QQ^+$ with $A$
finite, the set of those $G \in \G$ for which $h$ can be extended to
an automorphism of $S(U_G)$ is defined by  the conjunction of the
following Borel conditions:
\begin{itemize}
 \item $h$ is a partial automorphism of $S(U_G)$ (i.e.\ it is injective and respects the predicates $R, Q_r$ of $S(U_G)$);
\item for every $(n,q) \in A$, if there is $r \in \QQ^+$ such that
  $Q_r$ does not hold for $(n,q)$ in $S(U_G)$ then $(n,q) \, R\,
  h(n,q)\,  R \, (n,q)$ (in $S(U_G)$);
\item for every $(n,q) \in A$, if $Q_r$ holds for $(n,q)$ in $S(U_G)$
  for every $r \in \QQ^+$ and $h(n,q) = (n',q')$, then $n\, D_G \,
  n'$.
\end{itemize}
Indeed, fix $G \in \G$ and for $(n,q) \in \omega \times \QQ^+$ put $B_{n,q} =
B(\psi_n(U_G),q)$. Now notice that for each $(n,q) \in \omega \times
\QQ^+$ there are infinitely many distinct pairs $(n',q') \in \omega
\times \QQ^+$ such that $B_{n,q} = B_{n',q'}$ (so that in $S(U_G)$ we have
$(n,q)\, R\, (n',q')\, R\, (n,q)$ and $Q_r(n,q) \iff Q_r(n',q')$ for
every $r \in \QQ^+$): such pair will be called \emph{names} for the
basic ball $B_{n,q}$. If $h$ satisfies all the conditions above then it can
be extended to an automorphism $H$ of $S(U_G)$ in the following way.
Let $(n,q) \in A$ and let $(m,r) = h(n,q)$:
 if $B_{n,q}$ contains two distinct points, then $B_{n,q} = B_{m,r}$,
 while if $B_{n,q}$ isolate some point $x \in U_G$ then $B_{m,r}$
 isolate a point
 which belongs to the same fork of $x$ (hence $h$ induces partial
 permutations on the forks of $U_G$). Define $H$ on the names of
 $B_{n,q}$ by choosing any bijection extending $h$ between
such names and the names of $B_{m,r}$
 (this can be done because $A$ is finite and every basic ball
 have infinitely many
 names). For any other pair
 $(n',q')$ not yet considered (that is for any element of $\omega \times
 \QQ^+$ which is not the name of $B_{n,q}$ for some
 $(n,q) \in A$), we consider two cases: if $B_{n',q'}$ does not isolate a point
 of $U_G$ define $H(n',q') = (n',q')$. Otherwise, $B_{n',q'} = \{ x
 \}$ where $x$ belongs to some fork $F$ of $U_G$. Choose
 any permutation of the elements of $F$ which extends the partial
 permutation induced by $h$: if $x$ is sent to some $y$ by such
 permutation, then define $H$ on the names of $B_{n',q'}$ by choosing
 any permutation between such names and the names of the basic ball
 which isolate $y$. It is not hard to check that the $H$ constructed
 in this way is an
 automorphism of $S(U_G)$.
Conversely, if $H$ is an automorphism of $S(U_G)$  such that $h = H
\restriction A$, then the whole $H$ must satisfy the conditions above
because, as observed above, the isometry $H'$ of $U_G$ into itself
induced by $H$ must be the identity on its accumulation points and
set-wise fix all its forks: therefore $h$ satisfies such conditions as
well.
\end{proof}

\subsection{Linear isometric embeddings between separable Banach spaces}\label{banach}

Any
separable Banach space $X = (X, \| \cdot \|_X)$ is linearly isometric to
a closed subspace of
$\mathcal{C}([0,1])$ equipped with the sup norm, hence it is natural to
consider the Borel set $\mathcal{B}$ of
all closed linear subspaces of $\mathcal{C}([0,1])$ as the standard
Borel space of separable Banach spaces. For $X,Y \in \mathcal{B}$
say that $X$ \emph{linearly isometrically embeds} into $Y$ ($X
\sqsubseteq^{li} Y$ in symbols) if there is
a linear isometric embedding, that is a linear and norm-preserving map,
between $X$ and $Y$. The relation $\sqsubseteq^{li}$ is obviously an
analytic quasi-order, and the relation $\cong^{li}$ of \emph{linear
  isometry}  is the corresponding  $\sqsubseteq^{li}$-isomorphism.

In \cite{louros}
it was proved that $\sqsubseteq^{li}$ is a complete  analytic
quasi-order, so we would like to improve this result by showing that
$\sqsubseteq^{li}$ is indeed invariantly universal (coupled with
$\cong^{li}$).  In
order to apply
Theorem \ref{theorsaturation}, we need a Borel map $f$ simultaneously
reducing embeddability and isomorphism
on $\G$ to, respectively,  $\sqsubseteq^{li}$ and $\cong^{li}$, and
then a Borel
map $g$ reducing $\cong^{li}$ (possibly restricted to some standard
Borel space $Z$, invariant under linear isometry, with $f(\G) \subseteq Z \subseteq
\mathcal{B}$, so that
$\cong^{li}$-saturation in $Z$ will coincide with
$\cong^{li}$-saturation in $\mathcal{B}$)  to some orbit equivalence
relation in such a way that the map $\Sigma$ assigning to each $G \in \G$ the
stabilizer of $g(f(G))$ is Borel. A map $f$ with the required properties
was
already defined in \cite[Theorem 4.6]{louros}. Moreover, $\cong^{li}$
(on the whole $\mathcal{B}$)
is reducible to an orbit equivalence relation: this is because by a
theorem of Mazur $\cong^{li}$ and $\cong^i$ coincide on
$\mathcal{B}$, and by \cite{gaokechris} the relation $\cong^i$ is in
turn Borel reducible to
(in fact, Borel bireducible with)
the orbit
equivalence relation on $F(\mathbb{U})$ induced by the natural action
of the group of
automorphisms of the Urysohn space $\mathbb{U}$. However, the last
reduction is not quite explicit, and hence it seems to be  very
difficult to have a control on the resulting map $\Sigma$. Therefore
we will restrict our attention to a suitable proper
$\cong^{li}$-saturated
$Z \subseteq \mathcal{B}$, and to define such $Z$ we will first need
to slightly modify the
original construction of the map $f$ from \cite{louros}.

Let $c_0$ be the Banach space of sequences converging to $0$ endowed
with the sup-norm $\| \cdot \|_\infty$, let $\langle e_p \mid p \in
\omega \rangle$ be the usual base of $c_0$, and denote elements of
$c_0$ by $\Sigma_n \alpha_n e_n$. Given $G \in \G$ define\footnote{The
norm $\| \cdot \|_G$ (and hence the Banach space $X_G$) can actually be
defined starting from \emph{any} graph $G$, and not just from graphs
in $\G$. Moreover,
Lemma \ref{lemmareduction} would still hold in this more general
setup.} a new norm
$\| \cdot \|_G$ on $c_0$ by
\[ \| \Sigma_n \alpha_n e_n \|_G = {\rm sup} \left\{ |\alpha_i| +
\frac{|\alpha_j|}{3 - \chi_G(i,j)} \mid i \neq j \in \omega \right \}, \]
where $\chi_G \colon \omega \times \omega \to \{ 0,1 \}$ is the
characteristic function of the graph relation of
$G$. As for the norms defined in the proof of \cite[Theorem
4.6]{louros}, it is easy to check that $\| \cdot \|_G$ is indeed
equivalent to $\| \cdot \|_\infty$, as $\| \Sigma_n \alpha_n e_n
\|_\infty \leq \| \Sigma_n \alpha_n e_n \|_G \leq \frac{3}{2} \|
\Sigma_n \alpha_n e_n
\|_\infty$, and that the map $f$ sending $G \in \G$ into $X_G = (c_0,
\| \cdot \|_G)$ is
Borel (when we identify $(c_0, \| \cdot \|_G)$ with its linearly
isometric copy in $\mathcal{B}$).

\begin{lemma} \label{lemmareduction}
The map $f$ reduces $\sqsubseteq$ to $\sqsubseteq^{li}$ and $\cong$ to
$\cong^{li}$.
\end{lemma}

\begin{proof}
The proof is almost identical to that of \cite[Theorem
4.6]{louros}. If $h$ is an embedding of $G$ into $H$ then $\Sigma_n
\alpha_n e_n \mapsto \Sigma_n \alpha_n e_{h(n)}$ is a linear isometric
embedding of  $X_G$ into $X_H$.

Conversely, for any linear isometric embedding $h' \colon X_G \to
X_H$, one first proves (copying the original proof word by word) that
there is $h \colon \omega \to \omega$ such
that for every $p \in \omega$ there is $\epsilon_p \in \{ 1, -1\}$ for
which $h'(e_p) = \epsilon_p e_{h(p)}$. Arguing as in the original
proof one gets that for $p \neq q$ it is true that $h(p) \neq h(q)$
and moreover
\begin{multline*}
  1 + \frac{1}{3 - \chi_G(p,q)} = \| \epsilon_p e_p + \epsilon_q e_q
\|_G = \| h'(\epsilon_p e_p + \epsilon_q e_q)\|_H = \\
 = \| e_{h(p)} +
e_{h(q)} \|_H = 1 + \frac{1}{3 - \chi_H(h(p),h(q))},
\end{multline*}
so that $h$ is an embedding between $G$ and $H$ because $\chi_G(p,q) =
\chi_H(h(p),h(q))$.

As noted in \cite[Remark 4.7]{louros}, the above construction yields
also the result about isomorphism and linear isometry.
\end{proof}

It is not hard to check that in a Banach space of the form $X_G$ the
elements $e_p$ of the base and their opposite $-e_p$ are the unique
extreme points of the unit ball. Consider the set
\[
E =\{(x,X) \in \mathcal{C}([0,1]) \times \mathcal{B} \mid
x \text{ is an extreme point of the unit ball of } X\}.
\]
It is straightforward to check that $E$ is coanalytic, and it follows
from a result of Kaufman (\cite[Section III]{Kaufman}) that $E$ is not
Borel. This means that we cannot use $E$ to make our subsequent
computations work. Therefore we use the following stronger notion.

\begin{defin}
A point $x$ of the unit ball $B_X$ of the Banach space $X$ is
\emph{strongly extreme}  if and only if
\begin{equation} \label{eqse}
 \forall \varepsilon > 0\, \exists \delta > 0 \,\forall y,z \in B_X
 \left( \left\|
x - \frac{y+z}{2} \right\|_X \leq \delta \imp \| y-z \|_X \leq
\varepsilon \right).
\end{equation}
\end{defin}

Obviously, if $x \in X$ is a strongly extreme point of $B_X$ then it is
also an extreme point of $B_X$ (but the converse does not hold in
general). Moreover, it is easy to check using standard arguments that
\[
SE =\{(x,X) \in \mathcal{C}([0,1]) \times \mathcal{B} \mid
x \text{ is a strongly extreme point of the unit ball of } X\}.
\]
is Borel. In fact it is enough to check that $(x,X) \in SE$ if and only
if is satisfied the restricted version of \eqref{eqse} obtained by
considering just $\varepsilon,\delta  \in \QQ$ and $y,z$ ranging over
some countable dense subset of $B_X$.

\begin{lemma}\label{lemmase}
For any $G \in \G$, $\epsilon \in \{ -1,1\}$ and $p \in \omega$,
$\epsilon e_p$ is a strongly extreme point of  $B_{X_G}$.
\end{lemma}

\begin{proof}
Given $\varepsilon > 0$ put $\delta = \frac{1}{18} \varepsilon$. Let
$y = \Sigma_n \beta_n e_n,z = \Sigma_n \gamma_n e_n \in B_{X_G}$ be such that
$\| e_p - \frac{y+z}{2} \|_G \leq\delta $. The hypotheses on $y,z$ easily imply $1 - 2 \delta \leq
|\beta_p|, |\gamma_p| \leq 1$. But then  for $n \neq p$ we must have
$|\beta_n|,|\gamma_n| \leq 6 \delta$ since $y$ and $z$ both belong to
the unit ball, whence $|\beta_n - \gamma_n| \leq 12 \delta$ for every
$n \in \omega$. This implies that for every $i \neq j \in \omega$
\[ |\beta_i - \gamma_i| + \frac{|\beta_j - \gamma_j|}{3 - \chi_G(i,j)}
\leq 12 \delta + \frac{12 \delta}{2} =18\delta =\varepsilon ,\]
as required.
\end{proof}

Since all the points of $X_G$ which are not of the form $\epsilon
e_p$ are not even extreme points of $B_{X_G}$, from the previous lemma we can
conclude that $\{ \epsilon e_p \mid \epsilon \in \{ -1,1\}, p
\in \omega \}$ is the collection of all strongly extreme points of $B_{X_G}$.

Notice that the definition of $X_G$ in \cite{louros} is such that the unit ball has no strongly extreme points.
This is the reason why we had to change a bit such definition.

Now let $A$ be the closure under linear isometry of $f(\G)$ and
\[ B =
\{ X \in \mathcal{B} \mid \forall x,y ({SE(x,X) \wedge SE(y,X)} \imp
\|x-y\|_X > 1) \} . \]
Both $A,B$ are invariant with respect to $\cong^{li}$, $A$ is
analytic, $B$ is coanalytic,  and $A \subseteq B$ by direct computation of the values of $\|e_i - e_j
\|_G$ for $i \neq j \in \omega,G \in \G$. Therefore by
\cite[Exercise 14.14]{Kechris1995} there is a Borel set $C$ which is
$\cong^{li}$-invariant and such that $A \subseteq C \subseteq B$.
Notice that for each element of $C$ there are at most countably many strongly
extreme points of  its unit ball, as the same property clearly holds for all
elements of $B$ by separability of the Banach spaces under
consideration. This means that
the Borel set
\[ D = \{ (X,x) \in C \times \mathcal{C}([0,1]) \mid SE(x,X) \} \]
has countable vertical sections, so that by \cite[Exercise 18.15]{Kechris1995}
\[ C_\omega = \{ X \in C \mid
B_X\text{ has }\omega\text{-many strongly extreme points} \} \]
 is Borel as well and  there is a sequence $f_n \colon C_\omega \to
\mathcal{C}([0,1])$ of Borel functions such that $\langle f_n(X) \mid n \in
  \omega \rangle$ is an enumeration without repetitions of the strongly
extreme points of $B_X$ (for every $X \in C_\omega$).
Notice also that $C_\omega$ is obviously $\cong^{li}$-invariant and
contains $f(\G)$.

Now we further refine the set $C_\omega \subseteq \mathcal{B}$. Let
$\psi_n \colon C_\omega \to
\mathcal{C}([0,1])$ be a sequence of Borel functions such that for
every $X \in C_\omega$ the set $\{ \psi_n(X) \mid n \in \omega \}$ is
a dense subset of $X$. Given $X \in C_\omega$, put $X \in Z$ if and only if
for every $\varepsilon \in \QQ^+$ and every $m \in \omega$ there
  are $n \in \omega$, $\alpha_0 , \dotsc \alpha_n \in \QQ \setminus \{
  0 \}$, and $k_0, \dotsc, k_n \in \omega$ such that  $\|
  \psi_m(X) - (\alpha_0 f_{k_0}(X) + \dotsc + \alpha_n f_{k_n}(X)) \|_X
  \leq \varepsilon$, that is  if and only if the rational linear
  combinations of strongly extreme points of $B_X$ are dense in
$X$.

It is easy to check that $Z$  is Borel, $\cong^{li}$-invariant, and
contains $f(\G)$. The motivation in restricting our attention
to this particular $Z$ is that, as we will see, $\cong^{li}$ on $Z$ is
actually Borel
reducible to a relation of isomorphism on countable structures. Indeed, let $\Lambda$
be the language containing all $n+1$-ary relational symbols of the
form  $R^{\alpha_0,
  \dotsc, \alpha_n}_q$ for $n \in \omega$, $\alpha_0, \dotsc, \alpha_n
\in \QQ$, and $q \in \QQ^+$. Given $X \in Z$, define the
$\Lambda$-structure $S(X)$ on $\omega$ by stipulating that for each
$k_0, \dotsc, k_n \in \omega$  the predicate $R^{ \alpha_0, \dotsc,
  \alpha_n}_q(k_0, \dotsc, k_n)$ holds in $S(X)$ if and only if $\|
\alpha_0 f_{k_0}(X) + \dotsc + \alpha_n f_{k_n}(X) \|_X < q$.

\begin{lemma}\label{lemmaliiso}
 The function $g \colon Z \to Mod_\Lambda \colon X \mapsto S(X)$ is a
 Borel reduction of $\cong^{li}$ to $\cong$.
\end{lemma}

\begin{proof}
 Clearly $g$ is a Borel map. Let $X,Y \in Z$ be such that $X
 \cong^{li} Y$ via some $h$: then $h$ must be a bijection between the
 strongly extreme points of $B_X$ and $B_Y$, and therefore it
 induces in the obvious way a permutation $H$ of $\omega$ such that
 $h(f_n(X)) = f_{H(n)}(Y)$ for every $n \in
 \omega$. Moreover, since $h$ is linear
 and norm-preserving it is
 easy to verify that for every $\alpha_0, \dotsc, \alpha_n \in \QQ$, $q \in \QQ^+$ and $k_0, \dotsc, k_n \in \omega$
 one has that $R^{ \alpha_0, \dotsc, \alpha_n}_q(k_0, \dotsc, k_n)$
 holds in $S(X)$ if and only if $R^{ \alpha_0, \dotsc,
   \alpha_n}_q(H(k_0), \dotsc, H(k_n))$ holds in $S(Y)$, so that $S(X)
 \cong S(Y)$.

Conversely, let $H$ be an isomorphism between $S(X)$ and $S(Y)$, and
consider the linear extension $h$ of the map $f_n(X) \mapsto
f_{H(n)}(Y)$ to the sets of vectors which are rational linear
combinations of the extreme points of $B_X$ and $B_Y$ (denoted
respectively as $RLC_X$ and $RLC_Y$). To show that $h$ is well
defined, note first that it is norm-preserving, by the definition of
the predicates $R^{\alpha_0, \dotsc, \alpha_n}_q$ in $S(X)$ and
$S(Y)$. In particular, $h$ is well defined on the zero vector, and
its image is the zero vector. Let $\alpha_0, \dotsc, \alpha_n,
\beta_0, \dotsc , \beta_m \in \QQ, k_0, \dotsc, k_n, l_0, \dotsc, l_m
\in \omega$ be such that
$\sum_{i=0}^n\alpha_if_{k_i}(X)=\sum_{j=0}^m\beta_0 f_{l_j}(X)$. Then
$\sum_{i=0}^n\alpha_if_{H(k_i)}(Y)-\sum_{j=0}^m\beta_jf_{H(l_j)}(Y)$
is the image under $h$ of
$\sum_{i=0}^n\alpha_if_{k_i}(X)-\sum_{j=0}^m\beta_0 f_{l_j}(X)=0$. So
$\sum_{i=0}^n\alpha_if_{H(k_i)}(Y)-\sum_{j=0}^m\beta_jf_{H(l_j)}(Y)=0$.
Similarly, one can prove that $h$ is injective. Being surjective too,
$h$ is a bijection $RCL_X\to RCL_Y$.

By definition, $h$ is linear with
respect to rational linear combinations, i.e.\ if $v,w \in RLC_X$
and $\alpha, \beta \in \QQ$ then $h(\alpha v + \beta w) = \alpha h(v)
+ \beta h(w)$. But since $RLC_X$ and $RLC_Y$ are dense in $X$ and $Y$
(respectively) by the definition of $Z$, using standard
arguments we have that
$h$ extends to a (unique) linear isometry between $X$ and $Y$, so that
$X \cong^{li} Y$.
\end{proof}

Now we are ready to prove that $\sqsubseteq^{li}$ is invariantly universal.

\begin{theorem}
 The relation of linear isometric embeddability $\sqsubseteq^{li}$
 between separable Banach spaces is invariantly universal, that is for
 every analytic quasi-order $R$ there is a Borel class $\mathcal{C}
 \subseteq \mathcal{B}$ closed under linear isometry such that $R
 \sim_B \sqsubseteq^{li} \restriction \mathcal{C}$.
\end{theorem}

\begin{proof}
 We want to apply Theorem \ref{theorsaturation} to the maps $f$ and
 $g$ previously defined. Having already proved in Lemma
 \ref{lemmareduction} and Lemma
 \ref{lemmaliiso} that $f$ is as desired
 and $g$ reduces the linear isometry relation on $Z$ to the
 isomorphism relation, it remains only to show that the map $\Sigma$
 assigning to each $G \in \G$ the group of automorphisms of $S(X_G)$ is
 Borel. First notice that the map sending $G \in \G$ into the binary
 relation $O_G$ on $\omega$ defined by $n\, O_G\, m \iff f_n(X_G) = -
 f_m(X_G)$ is Borel. Given an injective sequence $s \in {}^{< \omega}
 \omega$, we now have that $S(X_G)$ has an automorphism extending $s$
 if and only if for every $i < | s|$, either $s_i = i$ or $s_i\, O_G
 \, i$: in fact, the proofs of Lemmas \ref{lemmareduction} and \ref{lemmaliiso}
 show that any automorphism $H$ of $S(X_G)$ such that $H(n) \neq n$
 and $\neg (H(n) \, O_G \, n)$ would
 induce a linear isometry $h$ of $X_G$ into itself such that $h(e_n) =
 \epsilon e_{H(n)}$ for $ \epsilon \in \{ 1,-1\}$, which in turn
 would naturally induce a nontrivial automorphism of $G$ sending $n$
 into $H(n)$, contradicting
 Corollary \ref{cornontrivial}. Therefore $\Sigma$ is Borel and we are
 done.
\end{proof}

\subsection{Further applications}\label{further}

In this section we will consider the remaining complete analytic quasi-orders which have appeared in the literature, and sketch the proofs that they are indeed invariantly universal.

Following \cite{Camerlo2005}, given two subcontinua $K,K'$ of Hilbert cube, we put:
\begin{itemize}
\item $K \preceq K'$ if there is a continuous surjection $K' \to K$;
\item $K \preceq_M K'$ if there is a monotone surjection $K' \to K$;
\item $K \preceq_R K'$ if there is an $r$-mapping $K' \to K$;
\item $K \preceq_W K'$ if there is a weakly confluent surjection $K' \to K$;
\item $K \preceq_O K'$ if there is an open continuous surjection $K' \to K$.
\end{itemize}

Clearly the isomorphism associated to each of these morphism relations is always the relation of homeomorphism $\simeq$. We start by considering the quasi-order $\preceq$ and give some hints on how to prove that the pair $(\preceq, \simeq)$ is indeed invariantly universal (we leave the details to the reader). In \cite[Theorem 3]{Camerlo2005} a Borel map $f$ was provided from countable graphs into compacta which reduces the embeddability relation to $\preceq$. As done in Subsection \ref{ref1}, it is not hard to check (using a back and forth argument for the nontrivial direction) that $f$ simultaneously reduces the isomorphism relation to $\simeq$.

In order to apply Theorem \ref{theorsaturation} we now would need to reduce
$\simeq$ to some orbit equivalence relation. Notice that, as observed
e.g.\ in \cite[Section 4.2]{louros}, $\simeq$ is reducible to the
natural action of the automorphism group of the Urysohn space
$\mathbb{U}$ on its closed subspaces (as by Banach-Stone $K \simeq K'
\iff (\mathcal{C}(K), \| \cdot \|_\infty) \cong^{li}
(\mathcal{C}(K'),\| \cdot \|_\infty)$). Unfortunately, with this reduction we lack control on the
complexity of the map assigning to each compacta the stabilizer of its
image under this reduction. Nevertheless, we can use an approach similar to
the one of Subsection \ref{banach}, i.e.\ we can restrict our attention to a
suitable Borel class $Z$ of compacta closed under homeomorphism and
containing $f(\G)$, and provide an \emph{ad hoc} reduction of ${\simeq}
\restriction Z$ to the isomorphism relation which satisfies the hypotheses of Theorem \ref{theorsaturation}.

However, in this particular case we can also use an even simpler argument to prove the invariant universality of $\preceq$: in fact, by the properties of Cook (sub)continua, the rigidity of each $G \in \G$ and the definition of the map $f$, we have that for every Borel $B \subseteq \G$, a compactum $K$ belongs to the $\simeq$-saturation of $f(B)$ if and only if there is a unique pair $(G,h)$ such that $G \in B$ and $h$ is an homeomorphism between $f(G)$ and $K$ (for more details we refer the reader to \cite{Camerlo2005}). Therefore the $\cong$-saturation of $f(B)$ is a Borel set by Luzin theorem \cite[Theorem 18.11]{Kechris1995} again, and as shown in the proof of Theorem \ref{theorsaturation} this fact implies that $(\preceq, \simeq)$ is invariantly universal.

Notice that this result automatically extends to the quasi-orders $\preceq_M$, $\preceq_R$ and $\preceq_W$ (see the explanation in \cite[p.\ 203]{Camerlo2005}), and even to any quasi-order $S$ on compacta such that ${{\preceq_M} \cap {\preceq_R}} \subseteq S \subseteq {\preceq}$ (when paired with $\simeq$). Moreover, it is not hard to show (applying Theorem \ref{theorsaturation} again) that $\preceq_O$ too is invariantly universal: this is because in \cite{Camerlo2005} it was proved that embeddability between graphs is reducible to $\preceq_O$ restricted to (very simple) dendrites, so that it is enough to combine the techniques developed in Subsections \ref{ref1} and \ref{dendrites} to get the desired result. Therefore we have the following:

\begin{theorem}
 The quasi-orders $\preceq,\preceq_M, \preceq_R, \preceq_W, \preceq_O$ are all invariantly universal. Moreover, for every quasi-order $S$ such that ${{\preceq_M} \cap {\preceq_R}} \subseteq S \subseteq {\preceq}$ the pair $(S, \simeq)$ is invariantly universal.
\end{theorem}

\section{Open problems}
We conclude the paper by stating a few questions about invariantly
universal analytic pairs which, in our opinion, deserve further
investigation.

\begin{question}\label{q:=}
Is the pair $(S,{=})$ invariantly universal whenever $S$ is a complete
analytic quasi-order?
\end{question}

Notice that if $S$ witnesses that the answer to the previous question
is negative, then $(S,E)$ is not invariantly universal for all analytic
equivalence relations $E$ (since ${=} \subseteq E$).

A more general question is the following:

\begin{question}\label{q:leq}
Suppose that $(S,F)$ is invariantly universal, $E \subseteq E_S$ and $E
\leq_B F$. Does it follow that $(S,E)$ is also invariantly universal?
\end{question}

The intuition behind Question \ref{q:leq} is that when $E \leq_B F$ it
should be easier to find $E$-invariant sets than to find $F$-invariant
sets.

The last questions (which are related) are necessarily less precise, as
they involve the notions of \emph{natural morphism relation} and
\emph{natural pair}. Here a ``natural'' morphism relation is any
morphism relation which is of independent interest in some area of
mathematics, while a natural pair would typically consist of a
quasi-order which is a natural morphism relation and of its isomorphism
relation (where both the morphism relation and its isomorphism relation are intended to be analytic).

\begin{question}
Is there a natural pair $(S,E)$ which is not invariantly universal,
although $S$ is a complete analytic quasi-order?
\end{question}

Recall that in the  introduction we gave an example of analytic quasi-orders $S$ and $R$ such that $(S,E_S)$ is invariantly universal while $(R,E_R)$ is not. The next question asks whether there exist (natural) morphism relations behaving in each of these ways.

\begin{question}\label{questionfinal}
Is there a natural morphism relation $S$ such that $(S,E_S)$ is invariantly universal? Similarly, is there a natural morphism relation $R$ which is complete analytic (or even invariantly universal when paired with some natural equivalence relation $E \subseteq E_R$) but such that $(R,E_R)$ is \emph{not} invariantly universal? In particular, if $S$ is the embeddability relation between countable graphs or one of the morphism relations considered in Section \ref{sectionapplications}, is it true that $(S,E_S)$ is invariantly universal?
\end{question}

Notice that Theorem \ref{theorsaturation} cannot be used to answer Question \ref{questionfinal}: if $S$ is complete for the class of analytic quasi-orders then $E_S$ is complete for the class of analytic equivalence relations, and hence it is not Borel reducible to an orbit equivalence relation.

\bibliography{cmmr}
\bibliographystyle{alpha}

\end{document}